\documentclass[11pt]{article}

\usepackage{amsmath}
\usepackage{amssymb}
\usepackage{amsthm}
\usepackage{latexsym}
\usepackage{color}
\usepackage{graphicx}
\usepackage{color}

\DeclareSymbolFont{calletters}{OMS}{cmsy}{m}{n}
\DeclareSymbolFontAlphabet{\mathcal}{calletters}

\newtheorem{Theorem}{Theorem}[part]
\newtheorem{Definition}{Definition}[part]
\newtheorem{Proposition}{Proposition}[part]
\newtheorem{Assumption}{Assumption}[part]
\newtheorem{Lemma}{Lemma}[part]
\newtheorem{Corollary}{Corollary}[part]
\newtheorem{Remark}{Remark}[part]

\newcommand{\nc}{\newcommand}
\nc{\esssup}{\mathop{\mathrm{ess\,sup}}}
\nc{\essinf}{\mathop{\mathrm{ess\,inf}}}
\nc{\argmax}{\mathop{\mathrm{arg\,max}}}

\def \P{\mathbb{P}}
\def \N{\mathbb{N}}

\def \R{\mathbb{R}}

\def \Q{\mathbb{Q}}

\def \1{\mathds{1}}

\def \l({{\left (}}
\def \r){{\right )}}
\def \l[{{\left [}}
\def \r]({{\right ]}}

\newcommand{\MBFigure}[6]{
$\left. \right.$ \\
\refstepcounter{figure}
\addcontentsline{lof}{figure}{\numberline{\thefigure}{\ignorespaces #5}}
\begin{center}
\begin{minipage}{#1cm}
\centerline{\includegraphics[width=#2cm,angle=#3]{#4}}
\begin{center}
\upshape{F\textsc{ig} \normal
\end{center}
size{\thefigure}. $-$} #5
\end{center}
\label{#6}
\end{minipage}
\end{center}
$\left. \right.$ \\}
 \title{Quadratic Exponential Semimartingales and Application to BSDE's  with jumps \footnote{Research supported by the Chair {\it Financial Risks} of the {\it Risk Foundation} sponsored by Soci\'et\'e G\'en\'erale, the Chair {\it Derivatives of the Future} sponsored by the {F\'ed\'eration Bancaire Fran\c{c}aise}, and the Chair {\it Finance and Sustainable Development} sponsored by EDF and Calyon.}
}
\author{ Nicole  {\sc El Karoui}\footnote{Laboratoire de Probabilit\'es et Mod\`eles Al\'eatoires, UPMC, nicole.elkaroui@cmap.polytechnique.fr.}
			\and
			Anis {\sc Matoussi} \footnote{Universit\'e du Maine, Institut  du Risque et de l'Assurance du Mans, anis.matoussi@univ-lemans.fr.}
      \and
      Armand  {\sc Ngoupeyou}\footnote{Banque Centrale des Etats de l'Afrique de l'Ouest (BCEAO), ngoupeyouarmand@yahoo.fr.}
      }
 \date{First version, May 4, 2014}

\begin{document}

 \maketitle

\begin{abstract}
In this paper,  we study a class of  Quadratic Backward Stochastic Differential Equations (QBSDE in short)   with jumps and unbounded terminal condition. We extend the class of quadratic semimartingales  introduced by Barrieu and El Karoui  \cite{BElK11}  in  the jump  diffusion model. The properties of these class of  semimartingales lead us to prove existence result for the solution of a quadratic BSDE's.
\end{abstract}
\section{Introduction}
Backward stochastic differential equations (in short BSDE's) were
first introduced by Bismut in 1973
\cite{Bismut} as equation for the adjoint process in the stochastic
version of Pontryagin maximum principle. 
Pardoux and  Peng \cite{PP90}  have generalized  the existence and uniqueness result in the case when the driver is Lipschitz  continuous.
Since then BSDE's have been widely used in stochastic control and
especially in mathematical finance, as any pricing problem by
replication can be written in terms of linear BSDEs, or non-linear
BSDEs when portfolios constraints are taken into account as in El Karoui,  Peng and Quenez \cite{ElkPQ97}.
Another direction which has attracted many works in this area,
especially in connection with applications, is how to improve the
existence/uniqueness conditions of a solution under weaker conditions on the driver.
Particularly in those papers it is
assumed that $f$ is just continuous and satisfies  a
quadratic growth condition. Among them we can quote Kobylanski \cite{K00},  Lepeltier and  San
Martin \cite{LepSan98} and so on. All of those works are assumed that the terminal condition is bounded and they are based on an exponential change of variable, troncation porcedure  and comparison theorem  of solutions
of BSDE's. Nonetheless, note that in general we do not have
uniqueness of the solution. In \cite{K00} a uniqueness result is also given by adding a more stronger conditions on the coefficient.
This latter model of BSDE's is very useful in mathematical finance
especially when we deal with exponential utilities or risk measure
theory especially weather derivatives (see e.g. El Karoui and  Rouge \cite{ElKarouiRouge00}, Mania and Schweizer \cite{MS05},
Hu, Imkeller and M\"uller \cite{HuImkellerMueller05},   Barrieu and El Karoui  \cite{BElkX08} and Becherer (\cite{Bech04}, \cite{Bech06})). Actually it has been shown in \cite{ElKarouiRouge00} that in a market model with constraints on the portfolios, the indifference price is given by the resolution of a BSDE with quadratic
growth coefficient. Finally let us point out that control
risk-sensitive problems turn into BSDE's which fall in the same
framework in  El Karoui and Hamad\`ene \cite{ElKarouiHamadene03}.
Our work was also motivated  by solving a utility maximization problem  of terminal wealth with  exponential utility function in models involving assets with jumps. Therefore we need to consider Backward Differential Equations with jumps  of the form
\begin{equation}\label{bsdej}Y_t=\eta_T+\int_t^T
f_s(Y_s,Z_s,U_s)ds-\int_t^T Z_s dW_s -\int_t^T \int_{E}U(s,x).\widetilde \mu(ds,dx)
\end{equation}
where $\widetilde\mu$ is a martingale random measure. {A solution of such  BSDE associated with $(f, \eta_T)$ is  a  triple of square integrable processes
${(Y_t,Z_t,U_t)}_{0\le t\le T}$. The standard BSDE's with jumps driven by Lipschitz coefficient was first introduced by Barles, Buckdahn and Pardoux \cite{BBP97} in order to give a probabilistic interpretation of viscosity solution of semilinear integral-Partial equations.  Afterwards the case of BSDE's with jumps and  quadratic coefficient was studied  by Becherer \cite{Bech06} and Morlais \cite{Mor07} in the context of exponential utility maximization problem in model involving jumps. In the both papers (\cite{Bech06}, \cite{Mor07}), the authors have used in the case of bounded terminal condition the Kobylanski method in the jump setting. As a consequence, they obtain that the state process $Y $ and the jump components $U$ of the BSDE solution are uniformly bounded, and that the martingale component is a BMO-martingale. Moreover, the so-called Kobylanski method is based on analytical point view inspired from Boccardo, Murat and Puel paper \cite{BMP83} and it is  based on the exponential change of variables,  troncation procedure and stability theorem. Therefore, one of the main difficulty in this method is  the proof of the strong convergence in the martingale part approximation. More recently Tevzadze \cite{MT06}  proposed a new different method to get the existence and uniquness of the solution of quadratic BSDE's. The method is based on a fixed point theorem but for only bounded terminal condition with small $L^{\infty}$-norm. \\
Our  main task in this paper is to deal with  quadratic BSDE's with  non-bounded terminal valued and  jumps. Our point of view  is inspired from Barrieu and El Karoui \cite{BElK11} for their study in the continuous case. By adopting a forward point of view, se shall characterize first a solution of  BSDE's as a quadratic It\^o semimartingale $Y$, with a decomposition satisfying the \textit{quadratic exponential structure condition} ${q}_{exp}(l,c,\delta)$, where the term \textit{exponential} refers to the exponential feature of the jump coefficient which appears in the generator of the BSDE. More precisely, we assume that: there exists nonnegative processes
constants $c$ , $\delta$ and $l$ such
\begin{equation}\label{eq1}
 -l_t-c_t|y|-{1\over 2}\delta |z|^2-{1\over \delta}j_t(-\delta u)
\leq  f(t,y,z,u) \le l_t+c_t|y|+{1\over 2}\delta |z|^2+{1\over
\delta}j_t(\delta u),  \; a.s.  
\end{equation}
where
$j_t(u)=\int_{E} (e^{u(x)}-u(x)-1)\xi(t,x)\lambda(dx)$. The canonical structure ${q}_{exp}(0,0,\delta)$ will play a essential role in the construction of the solution associated to  generale ${q}_{exp}(l,c,\delta)$ structure condition. The simplest
   generator of  a quadratic exponential BSDE,  called  the
\textit{canonical generator}, is defined as
$f(t,y,z,u)=q_{\delta}(z,u)={\delta\over 2}|z|^2 +{1\over
\delta}j(\delta u)$. For a given random variable $\psi_T$, we call  \textit{entropic process}, the process defined as $\rho_{\delta,t}(\eta_T)={1\over \delta}\ln \mathbb{E} \left[\exp(\delta \eta_T)\Big \vert {\cal F}_t\right]$ which is a  solution of the canonical BSDE's associated to the coefficient $q_{\delta}$ and final condition $\psi_T$. This is a entropic dynamic risk measure which have been studied, by  Barrieu  and El Karoui in \cite{BElkX08}. The backward point of view of our approach permits to relate the quadratic BSDEs to a quadratic exponential semimartingale with structure condition ${q}_{exp}(l,a,\delta)$, using the entropic processes. Namely,  a semimartingale $X$ with non bounded terminal condition $\eta_T$ and  satisfying the structure condition ${q}_{exp}(l,a,\delta)$, yields  the  following dominated inequalities $\rho_{-\delta,t}(\underline{U}_T)\le Y_t\le \rho_{\delta,t}(\overline{U}_T)$, where $\underline{U}_T$ and $\overline{U}_T$ are two random variable depending only on $l$, $a$, $ \delta$ and $\eta_T$. In the continuous seeting, 
Briand and Hu \cite{BriandHu05} prove implicitly the entropy inequalities  in the proof of the existence of the
solution of a quadratic BSDE, using Kobylanski method and localization procedure.\\
The  main goal in our approach is then to deduce, from this dominated inequalities,  a structure properties on the martingale
part and the finite variation part of $X$}.  Indeed, we obtain the canonical decomposition  of an entropic
quasimartingale which is a semimartingale which  satisfies the entropy inequalities;  as a canonical quadratic semimartingale part plus an predictable
increasing process.    This Doob type decomposition help us to define a general quadratic exponential
semimartingale as a limit of a sequence of canonical quadratic semimartingale plus a sequence of an increasing process. Then,
from the  stability theorem for forward semimartingales given by Barlow and Protter \cite{BP90}, we prove the existence of the solution of a
quadratic exponential BSDE associated to $(f, \eta_T)$ for  a coefficient $f$ satisfying the structure condition ${q}_{exp}(l,a,\delta)$ and for non-bounded terminal condition $\eta_T$. Finally, we have to mention that  it is  important to compare our approach with that used by Peng in \cite{Peng97, Peng03, Peng04}  within the representation theorem of small $g$-expectation  in terms of a BSDE's with  coefficient $g$ which  admits a linear growth condition in $z$. Peng's approach is based on the notion of martingale associated to a nonlinear expectation, Monotonic limit theorem, a nonlinear Doob-Meyer's decomposition Theorem (see e.g. \cite{Peng99}). Moreover, Peng obtained the representation theorem for  the  nonlinear expectation which is dominated by a structure nonlinear expectation solution of BSDE's with coefficient given specially by $ g_{\mu} (y,z) = \mu \big( |y| + |z|\big)$.  Barrieu  and El Karoui in \cite{BElkX08} have extended this representation theorem  for a dynamic convex risk measure in terms of quadratic BSDE's with convex coefficient $g$ which depends only in $z$. Our approach is an extension of the Peng's results in the more naturel framework of quadratic exponential semimartingale.\par\medskip
The paper is structured as follows: in a  second section, we  give a model and preliminary notation. In the third section, we define the quadratic exponential
semimartingale and we study the entropic quadratic exponential semimartingale. In particular, we give the characterization
of an entropic quasimartingale and its Doob decomposition. Then, a stability results of this class of semimartingale are given in the fourth section. The fifth  section is dedicated to  give 
 application of  the quadratic exponential semimartingale to prove existence result  for a class of QBSDE's associated to $(f, \eta_T)$  where the coefficient $f$ satisfies the structure condition ${q}(l,a,\delta)$ and for non-bounded terminal condition $\eta_T$. 

\section{Model and Preliminaries}
This section sets out the notations and the assumptions that are supposed in the sequel. We start with a stochastic basis $(\Omega,{\mathcal F},\mathbb{F},\mathbb{P})$ with finite horizon time $T<+\infty$ and a filtration $\mathbb{F}= \big(\mathcal F_t\big)_{t\in[0,T]}$ satisfying the usual conditions of right continuity and completness such that we can take all semimartingales to have right continuous paths with left limits. For simplicity, we assume ${\mathcal F}_0$ is trivial and ${\mathcal F}={\mathcal F}_T$. Without losing any generality we  shall work with a random measure to characterize   the jumps of any quasi-left continuous process $X$. Let $(E,{\cal E})$ a measurable space,  let define on the stochastic basis, a random measure left continuous $\mu(\omega,dt\times dx)$ on  $(\Omega,{\cal F})$ in $([0,T]\times {E},{\cal B}([0,T])\times {\cal E})$: 
$$\mu^X(\omega,dt,dx)=\sum_{s>0}1_{\{\Delta X_s(\omega)\not=0\}}{\varepsilon}_{({s,X_s(\omega))}}(dt,dx).$$
\noindent where $\varepsilon_a$ is the Dirac measure on $a$. Moreover the dual predictable projection $\nu^X$ of $\mu^X$ exists and it is called L\'evy system of $X$ (see Yor \cite{Yor75} for more details). For simplicity we note $\mu=\mu^X$ and $\nu=\nu^X$. Define the measure $\mathbb{P}\otimes \nu$ on $(\widetilde\Omega,\widetilde {\mathcal F})=(\Omega\times[0,T]\times E,{\cal F}\otimes {\cal B}([0,T])\otimes {\cal  E})$ by:
$$\mathbb{P}\otimes \nu(\widetilde B)=\mathbb{E}\left[\int_{[0,T]\times E} \textbf{1}_{\widetilde B}(\omega,t,e)\nu(\omega,dt,de)\right],\quad \widetilde B\in \widetilde {\cal F}.$$
\noindent Let ${\cal P}$ denote the predicatble $\sigma$-field on $\Omega \times [0,T]$ and define $\widetilde {\cal P}={\cal P}\otimes {\cal E}$, for any  $\widetilde {\cal P}$-measurable function $y$  with values in $\R$; we define:
$$y . \mu_t=\int_0^t\int_E y(w,s,x)\mu(w,ds,dx),\hbox{ and } y . \nu_t=\int_0^t\int_E y(w,s,x)\nu(w,ds,dx).$$
Let denote  by ${\cal G}_{loc}(\mu)$, the set of $\widetilde {\cal P}$- measurable functions $H$ with values in $\R$ such that $$|H|^2.\nu_t <\infty, \quad a.s.$$
Moreover if  $|H|.\nu_t<+\infty$ a.s, $H.\widetilde \mu:=H.(\mu-\nu)=H.\mu-H.\nu$  is a local martingale and we assume the following representation theorem for any local martingale $M$:   
$$M=M_0+M^c+M^d.$$ 
\noindent where $M^c$ is the continuous part of the martingale and $M^d$ is the discontinuous part, moreover there exists $U\in {\cal G}_{loc}(\mu)$ such that $M^d=U.(\mu-\nu).$ 
\\\\\noindent  We first introduce the following  notations:\\\\
 ${\cal M}^p_0$ is the set of martingale $M$ such that  $M_0=0$ and  $\mathbb{E}\left[\sup_{t\le T} |M_t|^p\right] <  +\infty$.\\
${\cal D}_{\exp}$ is the set of local  semimartingales $X$ such that $\exp(X) \in {\cal D}$ where $ \cal D$  (see \cite{Del72}, \cite{DM75} for the definition) \\
 ${\cal U}_{\exp}$ is the set of local martingales $M$ such that ${\cal E}(M)$ is uniformly integrable.

\section{Quadratic exponential semimartingales} 
In all our work, we shall consider the class of quasi-left continuous semimartingales $X$  with canonical decomposition $X=X_0-V+M$, with $V$ is a continuous predictable process with finite total variation $|V|$,   $M$ is a {c\`adl\`ag} local martingale satisfying the decomposition $M=M^c+M^d$  with $M^c$ is the continuous part of the  martingale $M$ and $M^d=U.\widetilde\mu$ for some $U\in {\cal G}_{loc}(\mu)$ is the purely discontinuous part .
The quadratic exponential semimartingales are the generalization of the quadratic semimartingales in jump diffusion models. The extra term in " exponential" comes from jumps and lead us to generalize the results given by \cite{BElK11}.

\begin{Definition}  The process $X$ is a local quadratic exponential semimartingale if there exists two positive continuous increasing processes $\Lambda$ and $C$ and a positive constant $\delta$ such that the processes $\delta M^c+(e^{\delta U}-1).\widetilde\mu $, $-\delta M^c+(e^{-\delta U}-1).\widetilde \mu$ are still  local  martingales and the the finite variation of $X$ satisfies the structure condition ${\cal Q}(\Lambda,C,\delta)$:
$$-{\delta\over 2}d\langle M^c\rangle_t-{1\over \delta}d\Lambda_t-|X_t|.dC_t-{1\over \delta}dj_t(-\delta \Delta M^d_t)<\!\!<dV_t<\!\!< {\delta\over 2}d\langle M^c\rangle_t+d\Lambda_t+|X_t|.dC_t+{1\over \delta}dj_t(\delta \Delta M^d_t)$$
The process $j(\gamma\Delta M^d)$ represents the predictable compensator of the increasing process  $A^\gamma:==\sum_{s\le t} (e^{\gamma\Delta M^d_s}-\gamma\Delta M^d_s-1)<+\infty$ a.s. $B<\!\!< A$ stands for $A-B$ is an increasing process.
\end{Definition} 
\begin{Remark} (About the dual predictable compensator)
\begin{itemize}
\item Before studying the properties of this class of local semimartingale, let first remark that for all $\gamma \in \{-\delta,\delta\}$, the increasing c\`adl\`ag process $j(\gamma \Delta M)$ is continuous applying Chap IV T[40] Dellacherie\cite{Del72}. Moreover using representation theorem of the discontinuous martingale  $M^d=U.\widetilde \mu$, then:
$$j_t(\gamma \Delta M^d_t)=(e^{\gamma U}-\gamma U-1).\nu_t.$$
\item Let remark that for $a=e^{\Delta U}-1$ , $b=e^{-\Delta U}-1$ since $-2ab\le a^2+b^2$, we find 
$$2[(e^{\delta U}-\delta U-1)+(e^{-\delta U}+\delta U-1)]\le |e^{\delta U}-1|^2+|e^{-\delta U}-1|^2$$
Since by assumption the processes $\delta M^c+(e^{\delta U}-1)\widetilde \mu $, $-\delta M^c+(e^{-\delta U}-1)\widetilde \mu$ are local martingales, the processes $|e^{\delta U}-1|^2.\nu_t$ a.s and $|e^{-\delta U}-1|^2.\nu_t<+\infty$ a.s,  therefore the  predictable compensator $j(\gamma \Delta M^d)$ of $A^{\gamma}$ is well defined for $\gamma\in \{-\delta,\delta\}$.
\end{itemize}
\end{Remark}
\noindent To understand better the class of local quadratic exponential semimartingales and theirs properties, we divide the class in three classes:
\begin{itemize}
\item \textit{The first class (The canonical class)}, where the finite variation part of $X$ satisfies:
$$V_t={1\over 2}\langle M^c\rangle_t+ j_t(\delta M^d_t) \hbox{ or } V_t=-{1\over 2}\langle M^c\rangle_t -j_t(\Delta M^d_t)$$
\item \textit{The second class (The class ${\cal Q}(0,0,1)$)}, where the finite variation part of $X$ satisfies:
$$-j_t(-\Delta M^d_t)-{1\over 2}\langle M^c\rangle_t<\!\!<V_t<\!\!< {1\over 2}\langle M^c\rangle_t+ j_t(\Delta M^d_t)$$
\item \textit{The third class (The general class} ${\cal Q}(\Lambda,C,\delta)$), where the finite variation part of $X$ satisfies:
$$-{\delta\over 2}\langle M^c\rangle_t-{1\over \delta}\Lambda_t-|X|*C_t-{1\over \delta}j_t(-\delta \Delta M^d_t)<\!\!<V_t<\!\!< {\delta\over 2}\langle M^c\rangle_t+{1\over \delta}\Lambda_t+|X|*C_t+{1\over \delta}j_t(\delta \Delta M^d_t)$$
\end{itemize}
\subsection{The canonical class}
\subsubsection{The exponential of Dol\'eans-Dade} 
We describe the relation between the exponential transform of a first class of local  quadratic exponential semimartingale and the exponential of Dol\'eans-Dade . Let first recall that for any {c\`adl\`ag} local semimartingale $X$, the exponential of Dol\'eans-Dade $Z$ of $X$ solving the EDS $dZ_t=Z_{t^-}dX_t$,    $Z_0=1$ is given by:
\begin{equation}\label{Doleans}
Z_t={\cal E}(X_t)=\exp(X_t-\langle X^c\rangle_t)\prod_{s\le t}(1+\Delta X_s)e^{-\Delta X_s}, \quad t\ge 0.
\end{equation}
\noindent This formula is given by the Ito's formula for discontinuous processes see Appendix (Theorem \ref{Ito1}, Corollary \ref{Ito2}) and Yor \cite{Yor75} for more details. We deduce that for a local martingale $M$, such that $\Delta M>-1$, the exponential of $M$ is a positive local martingale and there is some relation between  exponential of a canonical quadratic exponential semimartingale and  Dol\'eans-Dade of some local martingale.
\begin{Proposition}( Dol\'eans Dade martingale and canonical quadratic semimartingale). Let $\bar M=\bar M^c+\bar U.\widetilde \mu$ and $\underline{M}=\underline {M}^c+\underline{U}.\widetilde \mu$ two {c\`adl\`ag} local martingales such that  $\bar M^c+(e^{\bar U}-1).\widetilde \mu$ and 
$-\underline{M}^c+(e^{-\underline{U}}-1).\widetilde \mu$
are still {c\`adl\`ag} local martingales. Let define the canonical local quadratic exponential semimartingale:
\begin{equation*}
\begin{split}
&r(\bar M)=r(\bar M_0)+\bar M_t-{1\over 2}\langle \bar M^c\rangle_t-(e^{\bar U}-\bar U-1).\nu_t,\\
&\underline{r}(\underline{M})=\underline{r}(\underline{M}_0)+\underline{M}_t+{1\over 2}\langle \underline{M}^c\rangle_t+(e^{-\underline{U}}+\underline{U}-1).\nu_t
\end{split}
\end{equation*}
then we find the following processes:
$$\exp[r(\bar M)-r(\bar M_0)]={\cal E}\left(\bar M^c+(e^{\bar U}-1).\widetilde\mu\right) \hbox{ and  } \exp[-\underline{r}(\underline{M})+\underline{r}(\underline{M}_0)]={\cal E}\left(-\underline{M}^c+(e^{-\underline{U}}-1).\widetilde\mu\right)$$
\noindent are positive local martingales.
\begin{proof} We apply the Dol\'eans-Dade exponential formula \eqref{Doleans} with $\bar X=\bar M^c+(e^{\bar U}-1).\widetilde\mu$ and $\underline{X}=-\underline {M}^c+(e^{-\underline{U}}-1).\widetilde \mu.$ and we find the expected results.  
\end{proof}
\end{Proposition}
\begin{Definition}(${\cal Q}$- local martingale) A local semimartingale $X$ is a ${\cal Q}$-local martingale if $\exp(X)$ is a positive local martingale. 
\end{Definition}  
The canonical local quadratic exponential semimartingales $\bar r(\bar M)$ and $-\underline{r}(\underline{M})$ defined above are ${\cal Q}$- local martingales.
\subsubsection{The entropic risk measure}
The canonical local quadratic exponential semimartingales $\bar r(\bar M)$ and $\underline{r}(\underline{M})$  are ${\cal Q}$- local martingales, we can find more conditions on the local martingales $\bar M$ and $\underline{M}$ to get the uniform integrability condition of these semimartingales. Let first denote by ${\cal U}_{\exp}$ the set of local martingales $M$ such that ${\cal E}(M)$ is uniformly integrable and ${\cal D}_{exp}$ the set of local semimartingale such that $\exp(X) \in {\cal D}$. The sufficient condition that a local martingale $M=M^c+U.\widetilde \mu$ belongs to ${\cal U}_{\exp}$ is given in Lepingle and M\'emin \cite{LeMe78} Theorem IV.3:
\begin{equation}\label{UnifInt}
\begin{split}
\mathbb{E}\left[\exp\{{1\over 2}\langle M^c\rangle_{\tau}+\left((1+U)\ln(1+U)-U\right).\nu_\tau\}\right]<+\infty.
\end{split}
\end{equation}
 where $\tau=\inf\{t\ge 0, {\cal E}(M)=0\}$. This condition is sufficient and not necessary, another sufficient condition for a local semimartingale $X$ to belong to ${\cal D}_{exp}$ is satisfying if there exists a positive uniformly integrable martingale $M$ such that $\exp(X)\le M$. In particular theses sufficient conditions are satisfying for the dynamic entropic risk measure (see Barrieu and El Karoui for more details\cite{BElkX08}).
 \begin{Proposition} Let consider the fixed horizon time $T>0$ and $\psi_T \in {\cal F}_T$ such that $\exp(|\psi_T|)\in L^1$ and consider the two dynamic risk measures:
$$\bar\rho_t(\psi_T)=\ln\left[\mathbb{E}\left(\exp(\psi_T)\vert{\cal F}_t\right)\right],\hbox{ and } \underline{\rho}_t(\psi_T)=-\ln\left[\mathbb{E}\left(\exp(-\psi_T)\vert {\cal F}_t\right)\right]$$
There exists local martingales $\bar M=\bar M^c+\bar U.\widetilde \mu$ and $\underline{M}=\underline{M}^c+\underline{U}.\widetilde \mu$ such that:
\begin{equation*}
\begin{split}
&-d\bar\rho_t(\psi_T)=-d\bar M_t+{1\over 2}d\langle \bar M^c\rangle_t+\int_E(e^{{\bar U}(s,x)}-\bar U(s,x)-1).\nu(dt,dx),\quad \bar\rho_T(\psi_T)=\psi_T\\&
-d\underline{\rho}_t(\psi_T)=-d\underline{M}_t-{1\over 2}d\langle \underline{M}^c\rangle_t-\int_E(e^{-\underline{U}(s,x)}+\underline{U}(s,x)-1).\nu(dt,dx),\quad \underline{\rho}_T(\psi_T)=\psi_T
\end{split}
\end{equation*}
Moreover the local martingales $\bar M^c+(e^{\bar U}-1).\widetilde\mu$ and $-\underline{M}^c+(e^{-\underline{U}}-1).\widetilde\mu$ belong to ${\cal U}_{\exp}$. The dynamic risk measures  $\bar\rho(\psi_T)$ and $\underline{\rho}(\psi_T)$ are uniformly integrable canonical quadratic exponential semimartingales.
 \end{Proposition}
 \begin{proof} We have $\exp(\bar\rho_t(\psi_T))=\mathbb{E}\left[\exp(\psi_T)\vert {\cal F}_t\right]$ which is a positive uniform integrable martingale since $\exp(|\psi_T|)\in L^1$ then there exists a martingale $\bar X\in {\cal U}_{exp}$ satisfying $\Delta \bar X>-1$ such that $\exp(\bar\rho_t(\psi_T))={\cal E}(\bar X_t)$. Using martingale representation Theorem there exists a continuous martingale $M^c$ and a process $U$ satisfying $e^U-1\in {\cal G}_{loc}(\mu)$ such that $\bar X=\bar M^c+(e^{\bar U}-1).\widetilde \mu.$ Therefore we find
$\exp(\bar\rho_t(\psi_T))={\cal E}(\bar M^c_t+(e^{\bar U}-1).\widetilde\mu_t)=\exp(\bar r(\bar M_t))$. We use the same arguments to prove that there exists a martingale $\underline{X}=-\underline{M}^c+(e^{-\underline U}-1).\widetilde\mu\in {\cal U}_{\exp}$ such that
$\exp(-\underline{\rho}_t(\psi_T))={\cal E}(-\underline{M}^c_t+(e^{-\underline U}-1).\widetilde\mu_t)=\exp(-\underline{r}(\underline{M}_t)).$
\end{proof}
We adopt a forward and backward points of view to describe the canonical local quadratic exponential semimartingales class. In the forward point of view, we give condition of some martingales using Dol\'eans-Dade exponential formula to find that for any canonical local  quadratic exponential semimartingale $X$, $\exp(X)$ or $\exp(-X)$ is a local martingale. In the backward point of view, we fix a terminal condition $X_T \in {\cal F}_T$ such that $\exp(|X_T|)\in L^1$, then we can prove that some dynamic entropic risk measures of $\psi_T$ belongs to canonical quadratic exponential semimartingale class. In this point of view, we do not  make assumption on the martingale part of the canonical semimartingale to satisfy the Lepingle and M\'emin condition \eqref{UnifInt} since the exponential condition on the terminal condition is sufficient to find uniform integrability condition. 
\subsection{The second class: ${\cal Q}(0,0,1)$}
\subsubsection{The exponential transform}
In the first part, we use the Dol\'eans-Dade formula to explain how the canonical local quadratic exponential semimartingale can be represented using an exponential transform. The same technics can be developped for ${\cal Q}(0,0,1)$- local semimartingale using the multiplicative decomposition Theorem studied by Meyer and Yoeurp \cite{{MeYoe76}}) which stands that for any c\`adl\`ag  positive local submartingale $Z$ there exists an predictable increasing process $A$ ($A_0=0)$ and a local martingale
$M$ ($\Delta M>-1, M_0=0$) such that:
$$Z_t=Z_0\exp(A_t).{\cal E}(M_t),\quad t\ge 0.$$
\begin{Theorem}\label{expsubmartingale} Let $X$ a l\`adl\`ag process, $X$ is a ${\cal Q}(0,0,1)$-local semimartingale if and only if $\exp(X)$ and $\exp(-X)$ are local submartingales. In both cases, $X$ is a c\`adl\`ag process.
\end{Theorem}
\begin{proof} Let consider a ${\cal Q}(0,0,1)$- local semimartingale $X$ with canonical decomposition $X=X_0-V+M$ where $V$ is the finite variation part of $X$ (continuous) and $M$ is a local martingale, then there exists $U\in {\cal G}_{loc}(\mu)$ such that $M=M^c+U.\widetilde \mu$. Applying Ito' s formula, we find the decomposition of $\bar Z=\exp(X)$:
$$d \bar Z_t=\bar Z_{t^-}\left[dM^c_t+\int_E (e^{U(t,x)}-1).\widetilde \mu(dt,dx)-dV_t+{1\over 2}d\langle M^c\rangle_t+\int_E (e^{U(t,x)}-U(t,x)-1)\nu(dt,dx)\right]$$
\noindent Since $X$ is a ${\cal Q}(0,0,1)$- semimartingale then $A=-V+{1\over 2}\langle M^c\rangle+(e^U-U-1).\nu$ is an increasing continuous predictable process. Therefore the process $Z=\exp(X)$ is a positive local submartingale and satisfies the following Meyer and Yoeurp multiplicative decomposition:
$$\exp(X_t-X_0)=\exp(A_t){\cal E}( M^c_t+(e^U-1).\widetilde \mu_t),\quad t\ge 0.$$
\noindent  We use the same arguments to prove that $\exp(-X)$ is a local positive submartingale. Let now assume that $\exp(X)$ and $\exp(-X)$ are local submartingales where $X$ is a l\`adl\`ag process. Using Meyer and Yoeurp multiplicative decomposition, there exist local martingales $\bar M, \underline{M}$ and increasing predictable processes $\bar A,\underline{A}$ such that $\exp(X_t-X_0)=\exp(\bar A_t){\cal E}(\bar M_t)$ and $\exp(-X_t+X_0)=\exp(\underline{A}_t){\cal E}(\underline{M}_t)$. Using the representation martingale Theorem, there exist $\bar U, \underline{U}\in{\cal G}_{loc}(\mu)$ and continuous local martingales $\bar M^c, \underline{M}^c$ such that $\bar M=\bar M^c+(e^{\bar U}-1).\widetilde \mu$ and $\underline{M}=\underline{M}^c+(e^{\underline {U}}-1).\widetilde \mu.$ Hence we find $\exp(X-X_0)=\exp(\bar A)\exp(r(\bar M)$ and $\exp(-X+X_0)=\exp(\underline{ A})\exp(r(\underline{M})$ and we get
$$X_t-X_0=\bar A_t+\bar M_t-{1\over 2}\langle \bar M^c\rangle-(e^{\bar U}-\bar U-1).\nu_t  \hbox{  and  } -X_t+X_0=\underline {A}_t+\underline{M}_t-{1\over 2}\langle \underline{M}^c\rangle-(e^{\underline{U}}-\underline{U}-1).\widetilde \nu_t.$$
 Using the uniqueness of the representation of the semimartingale $X$, we deduce, $\underline{M}=-\bar M$, then we find $\bar A_t+\underline{A}_t=\langle \bar M^c\rangle_t+(e^{\bar U}-\bar U-1).\nu_t+(e^{-\bar U}+\bar U-1).\nu_t$ . The process $\bar A$ and $\underline A$ are continuous, moreover from Radon Nikodym's Theorem, there exists a predictable process with $0\le \alpha_t\le 2$ such that $d\bar A_t={\alpha_t\over 2}d\left[\langle \bar M^c\rangle_t+(e^{\bar U}-\bar U-1).\nu_t+(e^{-\bar U}+\bar U-1).\nu_t\right]$. Therefore the process $X$ satisfies the dynamics $dX_t=dM_t-dV_t$ where:
$$dV_t={(1-\alpha_t)\over 2}d\langle M^c\rangle_t+{(2-\alpha_t)\over 2}d\left[(e^{\bar U}-\bar U-1).\nu_t\right] -{\alpha_t\over 2}d\left[(e^{-\bar U}+\bar U-1).\nu_t\right]$$  
\noindent Since $0\le \alpha_t\le 2$, the local semimartingale $X$ satisfies the structure condition  ${\cal Q}(0,0,1)$. Moreover the finite variation part $V$ of $X$ is a predictable continuous process. We deduce $X$ is ${\cal Q}(0,0,1)$- local semimartingale and that all jumps of $X$ come from the local  martingale part which is c\`adl\`ag process.  
\end{proof}
\begin{Definition} Let consider a local semimartingale $X$, if $\exp(X)$ is a local submartingale then $X$ is called ${\cal Q}$- local submartingale. 
\end{Definition}
From Theorem \ref{expsubmartingale}, any ${\cal Q}(0,0,1)$- local semimartingale is a ${\cal Q}$- local submartingale and the reverse holds true.
\subsubsection{The entropic submartingales}
We are interested to find uniform integrability condition for ${\cal Q}(0,0,1)$- local semimartingales. Since ${\cal Q}(0,0,1)$- local semimartingales are ${\cal Q}$- local submartingales, we use the same technics developped for standard local submartingales. We recall that to prove $X\in {\cal D}_{exp}$, it is sufficient to prove there exists a positive martingale $L\in{\cal D}$ such that  $\exp(X)\le L$. To construct the positiive martingale $L$, let first give some useful definitions.
\begin{Definition} A process $X \in {\cal D}_{exp}$ is called an entropic submartingale if for any stopping times $\sigma\le \tau$:
$$X_\sigma\le \bar\rho_\sigma(X_\tau),\quad \sigma\le \tau.$$
\noindent where $\bar \rho$ stands for the usual entropic risk measure defined above. In the same point of view, $X$ is called a entropic supermartingale if $-X$ is an entropic submartingale. If $X$ and $-X$ are entropic submartingales, $X$ is called entropic quasi-martingale.
\end{Definition}   
\begin{Theorem}\label{uniformQexp} Let $T>0$ the fixed horizon time and consider a semimartingale $X=X_0-V+M^c+U.\widetilde \mu$  such that $\exp(|X_T|)\in L^1$ then $X$ is a ${\cal Q}(0,0,1)$-semimartingale $\in {\cal D}_{exp}$  if and only if $X$ and $-X$ are entropic submartingales.Moreover, in all cases the martingales $M^c+(e^{U}-1).\widetilde \mu$ and $-M^c+(e^{-U}-1).\widetilde \mu$ belong to ${\cal U}_{exp}$.
\end{Theorem}
\begin{proof} Let consider a ${\cal Q}(0,0,1)$-semimartingale $X=X_0+M^c+U.\widetilde\mu-V\in{\cal D}_{exp}$ such that $\exp(|X_T|)\in L^1$. Since $X$ is ${\cal Q}$- submartingale we find:
$$\exp(X_t)\le \mathbb{E}\left[\exp(X_T)\vert {\cal F}_t\right]\in {\cal D} \hbox{ and } \exp(-X_t)\le \mathbb{E}\left[\exp(-X_T)\vert {\cal F}_t\right]\in {\cal D}$$
\noindent and for any stopping times: $\sigma\le \tau\le T:$ $$X_\sigma\le \ln\left(\mathbb{E}\left[\exp(X_\tau)\vert {\cal F}_\sigma\right]\right)=\bar\rho_\sigma(X_\tau) \hbox{ and }-X_\sigma\le \ln\left(\mathbb{E}\left[\exp(-X_\tau)\vert {\cal F}_\sigma\right]\right)=\bar\rho_\sigma(-X_\tau).$$
\noindent then $X$ and $-X$ are entropic submartingales. Let prove the reverse, assume $X$ and $-X$ are entropic submartingales then for any stopping times $\sigma\le \tau$, $\exp(X_\sigma))\le \mathbb{E}\left[\exp(X_\tau)\vert {\cal F}_\sigma\right]$ and $\exp(-X_\sigma))\le \mathbb{E}\left[\exp(-X_\tau)\vert {\cal F}_\sigma\right]$, then $X$ is a uniformly integrable ${\cal Q}$-submartingale and from Theorem \ref{expsubmartingale}, $X$ is a ${\cal Q}(0,0,1)$-semimartingale. Since $X$ and $-X$ belong to ${\cal D}_{exp}$ then for a fixed horizon time $T$, $\exp(X_T)$ and $\exp(-X_T)$ belong to $L^1$ which lead to conclude $\exp(|X_T|)\in L^1$.  Moreover since $X$ and $-X$ are ${\cal Q}$-submartingales using Meyer-Yoeurp multiplicative decomposition Theorem, there exist increasing processes $\bar A$ and $\underline{A}$ ($\bar A_0=0$ and $\underline{A}_0=0)$)  such that:
$$\exp(X_t-X_0)=\exp(\bar A_t){\cal E}(M^c_t+(e^U-1).\widetilde \mu_t) \hbox{ and } \exp(-X_t+X_0)=\exp(\underline{A}_t){\cal E}(-M^c_t+(e^{-U}-1).\widetilde \mu_t).$$ \noindent Therefore we deduce that
${\cal E}(M^c_t+(e^U-1).\widetilde \mu)\le \exp(X_t-X_0) \hbox{ and } {\cal E}(-M^c_t+(e^{-U}-1).\widetilde \mu)\le \exp(-X_t+X_0).$
Since $|X-X_0| \in {\cal D}_{exp}$, we conclude the martingales $M^c+(e^U-1).\widetilde \mu$ and $-M^c+(e^{-U}-1).\widetilde \mu\in {\cal U}_{exp}$.        
\end{proof}
To conclude this part, we can make some links with the sublinear $g$-expectation of Peng \cite{PeXu05} since if we define the $g$-expectation of $X$ by $\mathbb{E}^g(X)$, we can define the submartingale under the $g$-expectation. Therefore, we deduce that if $X$ is ${\cal Q}(0,0,1)$-semimartingale such that $|X|\in{\cal D}_{exp}$, $X$ and $-X$ are submartingales under $\mathbb{E}^g=\ln\left[\mathbb{E}(\exp)\right]$.  
\subsection{ General class:${\cal Q}(\delta,\Lambda,C)$}
\subsubsection{The exponential transform}
We use some exponential transformations for general ${\cal Q}(\Lambda,C,\delta)$ local quadratic exponential semimartingale  such that the new tansformed process belong to the class ${\cal Q}(0,0,1)$. Therefore, we can apply the same methodology using in the previous sections to find general results for  ${\cal Q}(\Lambda,C,\delta)$ local semimartingales.
\begin{Proposition}\label{transform} Let consider a ${\cal Q}(\Lambda,C,\delta)$-local semimartingale $X=X_0-V+M^c+U.\widetilde\mu$ then
\begin{enumerate}
\item For any $\lambda\not=0$, the process $\lambda X$ is a ${\cal Q}(\Lambda,C,{\delta\over |\lambda|})$-local semimartingale and a ${\cal Q}(\lambda\Lambda,C,\delta)$-local semimartingale when $\lambda>1$.
\item Let define the two transformations:
$$Y^{\Lambda,C}(X)=X+\Lambda+|X|*C  \hbox{ and } \bar Y^{\Lambda,C}(|X|)=e^C|X|+e^C*\Lambda.$$
\noindent then the two processes $Y^{\Lambda,C}(\delta X)$ and  $\bar Y^{\Lambda,C}(|\delta X|)$ are ${\cal Q}$-local submartingales.
\item Exponential transformation: Let $U^{\Lambda,C}(X)$ the transformation
$$U^{\Lambda,C}_t(e^X)=e^{X_t}+\int_0^t e^{X_s}d\Lambda_s+\int_0^t e^{X_s}|X_s|dC_s.$$ 
then $U^{\Lambda,C}(e^{\delta X})$ is a positive local submartingale.
\end{enumerate}
\end{Proposition}  
\begin{proof}
\begin{enumerate}
\item Let consider a ${\cal Q}(\Lambda,C,\delta)$-local semimartingale $X=X_0-V+M^c+M^d$ (where $M^d=U.\widetilde \mu$) and consider $\lambda \not=0$, hence $\lambda X=\lambda X_0 -\lambda V+\lambda M^c +\lambda M^d$ and  $\lambda X$ satisfies the condition
\begin{equation*}
 \left\lbrace
\begin{split}
& -|\lambda|{\delta\over 2}d\langle M^c\rangle_t-{|\lambda|\over \delta}d\Lambda_t-|\lambda X_t|.dC_t-|\lambda|{1\over \delta}dj_t[-\delta \hbox{ sign }(\lambda)\Delta M^d_t]<\!\!<\lambda dV_t,\\
& \lambda dV_t<\!\!< |\lambda|{\delta\over 2}d\langle M^c\rangle_t+{|\lambda|\over \delta} d\Lambda_t+|\lambda X_t|.dC_t+|\lambda|{1\over \delta}dj_t[\delta  \hbox {sign} (\lambda) \Delta M^d_t].
\end{split}
\right.
\end{equation*}
 Since $j(\delta\Delta M^d)=j[{\delta\over \lambda}(\lambda \Delta M^d]$ then we find
\begin{equation*}
 \left\lbrace
\begin{split}
& -{\delta\over |\lambda|}{1\over 2}d\langle \lambda M^c\rangle_t-{|\lambda|\over \delta}d\Lambda_t-|\lambda X_t|.dC_t-{|\lambda|\over \delta}dj_t[-{\delta\over |\lambda|}(\lambda \Delta M^d_t)]<\!\!<\lambda dV_t,\\
& \lambda dV_t<\!\!< {\delta \over |\lambda|}{1\over 2}d\langle \lambda M^c\rangle_t+{|\lambda|\over \delta} d\Lambda_t+|\lambda X_t|.dC_t+{|\lambda|\over \delta}dj_t[{\delta\over |\lambda|}(\lambda \Delta M^d_t)].
\end{split}
\right.
\end{equation*}

\noindent then $\lambda X$ is a ${\cal Q}(\Lambda,C,{\delta\over |\lambda|})$-local semimartingale. Moreover for $\lambda>1$: 
$${\delta \over |\lambda|}{1\over 2}d\langle \lambda M^c\rangle_t<\!\!< {\delta\over 2}d\langle M^c\rangle_t \hbox{ and  }{|\lambda|\over \delta}j_t[{\delta\over |\lambda|}(\lambda \Delta M^d_t)]<\!\!< {1\over \delta}j_t[\delta (\lambda \Delta M^d)]$$
 see Lemma \ref{ineqsaut} in Appendix for more details for this inequality. We find that for $\lambda>1$, $\lambda X$ is a ${\cal Q}(|\lambda|\Lambda,C,\delta)$-semimartingale.
\item Let consider the $Y^{\Lambda,C}(\delta X)=\delta X_0+\widetilde M_t-\widetilde V_t$, where $\widetilde M$ is the local martingale part given by $\widetilde M=\delta M^c+\delta M^d$ and $\widetilde V$ the finite variation part given by $\widetilde V=\delta V-\Lambda-|\delta X|*C$ . Since $X$ is ${\cal Q}(\Lambda,C,\delta)$-local semimartingale we have $d\widetilde V_t<\!\!< dj_t(\delta \Delta M^d)+{\delta^2\over 2}d\langle M^c\rangle_t$. We conclude $d\widetilde V_t<\!\!<dj_t(\Delta \widetilde M^d)+{1\over 2}d\langle \widetilde M^c\rangle_t$ and the process $A$ defined by $dA_t=-d\widetilde V_t+dj_t(\Delta \widetilde M^d)+{1\over 2}d\langle \widetilde M^c\rangle_t$ is an increasing process and $Y^{\Lambda,C}(\delta X)=\delta X_0+\widetilde M-{1\over 2}\langle \widetilde M^c\rangle-j(\Delta \widetilde M^d)+A$ then we conclude $\exp(Y^{\Lambda,C}(\delta X))$ is a local submartingale then it is ${\cal Q}$-local submartingale. Let prove now the process $\bar Y^{\Lambda,C}(|X|)$ belong to  the ${\cal Q}(0,0,1)$-class, let applying the Meyer-Ito's formula, we find the decomposition: $$de^{C_t}|X_t|=e^{C_t}\left[|X_t|dC_t-\hbox{sign}(X_{t^-})dV_t+dL^X_t+d\left[(|X_{-}+U|-|X_{^-}|).\nu_t\right]+d\bar M_t\right]$$ \noindent where $d\bar M_t=\hbox{sign}(X_{t^-})dM^c_t+d\left[(|X_{-}+U|-|X_{^-}|).\widetilde \mu_t\right]$ and $L^X$ stands for the local time of $X$ at $0$. Therefore the decomposition of the semimartingale $\bar Y^{\Lambda,C}([\delta X|)$ satisfies $d\bar Y^{\Lambda,C}(|X|)=-d\widetilde V_t+d\widetilde M_t$ where  $\widetilde M=\delta \bar M$ and $$d \widetilde  V_t=-e^{C_t}\left[|\delta X_t|dC_t+d\Lambda_t-\delta\hbox{sign}(X_{t^-})dV_t+dL^{\delta X}_t+d\left[\delta(|X_{-}+U|-|X_{^-}|).\nu_t\right]\right].$$
\noindent Since the process $X$ is a ${\cal Q}(\Lambda,C,\delta)$-local semimartingale, the process $A$ defined by $dA_t=\delta(|X_t|dC_t+{1\over \delta}d\Lambda_t-\hbox{sign}(X_{t^-})dV_t+{\delta\over 2}d\langle M^c\rangle_t+{1\over \delta}dj_t[\hbox{sign}(X_{t^-})]\delta |\Delta M^d|)+{1\over \delta}dL^{\delta X}_t$) is an increasing process. Therefore we get:
\begin{equation*}
\begin{split}
-d\widetilde V_t= e^{C_t}\left[-{\delta^2\over 2}d\langle M^c\rangle_t-dj_t[\delta \hbox{sign}(X_{t^-})\Delta M_t]+d\left(\delta(|X_{-}+U|-|X_{^-}|).\nu_t\right).\right]
\end{split}
\end{equation*}
From Lemma \ref{ineqsaut} (see Appendix for details), for any $k\ge 1$, $j(k \Delta M)\ge kj(\Delta M)$, therefore since $C$ is an increasing process with the initial condition $C_0=0$, we get $j_s[\delta e^{C_s}\hbox{sign}(X_{s^-})\Delta M_s]-e^{C_s}j_s[\delta \hbox{sign}(X_{s^-})\Delta M_s]\ge 0$. Moreover for any $s\ge 0$, ${\delta^2\over 2}\langle e^{C_s}M^c\rangle_s-{\delta^2\over 2}e^{C_s}\langle M^c\rangle_s \ge 0$, then   we obtain:
$$-d\widetilde V_t=-{1\over 2}d\langle e^{C_t}\delta\hbox{sign}(X_{t^-})M^c\rangle_s-dj_t[\delta e^{C_t}\hbox{sign}(X_{t^-})\Delta M_t]+d\left(\delta(|X_{-}+U|-|X_{^-}|).\nu_t\right)+d\bar A_t$$
\noindent where $\bar A$ is an increasing process.Finally we get:
\begin{equation*}
\begin{split}
d\bar Y^{\Lambda,C}_t(|X|)&=e^{C_t}\delta\hbox{sign}(X_{t^-})dM^c_t+\int_{E}e^{C_t}\delta(|X_{t^-}+U(t,x)|-|X_{t^-}|)\widetilde \mu(dt,dx)\\&
-{1\over 2}d\langle e^{C_t}\delta\hbox{sign}(X_{t^-})M^c\rangle_t-j_t[\delta e^{C_t}(|X_{t^-}+\Delta M_t|-|X_{t^-}|)]
+d\widetilde A_t
\end{split}
\end{equation*}
\noindent where
$$\widetilde A_t=\bar A_t+\int_0^t\int_{E}\left[\exp\left(e^{C_s}\delta(|X_{s^-}+U(s,x)|-|X_{s^-}|)\right)-\exp\left(e^{C_s}\hbox{sign}(\delta X_{s^-})U(s,x)\right)\right]\nu(ds,dx)$$
\noindent Since $|y+u|-|y|\ge \hbox{sign}(y)u$ we deduce $\widetilde A$ is increasing then we get: 
$$\bar Y^{\Lambda,C}(|X|)=|\delta X_0|+\widetilde M-{1\over 2}\langle \widetilde M\rangle-j(\Delta \widetilde M)+\widetilde A.$$
\noindent Therefore, $\exp(\bar X^{\Lambda,C})$ is a local submartingale then it is ${\cal Q}$-local submartingale.
\item Let apply Ito's formula to find the decomposition of $U^{\Lambda,C}(e^{\delta X})$:
$$dU^{\Lambda,C}_t(e^{\delta X})=e^{\delta X_{t^-}}\left[\delta dM^c_t+d[(e^{\delta U}-1).\widetilde \mu_t]-\delta dV_t+{\delta^2\over 2}d\langle M^c\rangle_t+dj_t(\delta \Delta M_t)+|\delta X_t|dC_t\right].$$
\end{enumerate}
\noindent Since $X$ is ${\cal Q}(\Lambda,C,\delta)$-local semimartingale then the process $A$ defined by 
$dA_t=-dV_t+{\delta\over 2}d\langle M^c\rangle_t+{1\over\delta}dj_t(\delta\Delta M_t)+|\delta X_t|dC_t$ is an increasing process, we deduce the process $U^{\Lambda,C}(e^{\delta X})$ is a positive local submartingale. 
\end{proof}
\begin{Theorem}\label{generalsubmartingale} Let $X$ a l\`adl\`ag optionnal process $X$. $X$ is a ${\cal Q}(\Lambda,C,\delta)$-local semimartingale if and only if $\exp\left[Y^{\Lambda,C}(\delta X)\right]$ and 
$\exp\left[Y^{\Lambda,C}(-\delta X)\right]$ are submartingales or equivently if the processes $U^{\Lambda,C}(e^{\delta X})$ and  $U^{\Lambda,C}(e^{-\delta X})$ are local submartingales. In all cases; $X$ is a c\`adl\`ag process.
\end{Theorem}
 \begin{proof} Let consider a ${\cal Q}(\Lambda,C,\delta)$-local semimartingale $X$, using Proposition \ref{transform}-2), we prove the process $\exp(Y^{\Lambda,C}(\delta X))$ is a local submartingale. The same arguments lead us to conclude also that $\exp(Y^{\Lambda,C}(-\delta X))$ is a local submartingale since $-X$ as the same structure condition as $X$. Let now consider that the both processes $\exp(Y^{\Lambda,C}(\delta X))$ and $\exp(Y^{\Lambda,C}(-\delta X))$  are positive submartingales then we can apply the Yoeurp-Meyer decomposition as Theorem \ref{expsubmartingale} and conclude there exists continuous local martingales $\bar M^c,\underline{M}^c$, increasing processes $\bar A, \underline{A}$ and $\bar U, \underline{U} \in {\cal G}_{loc}(\mu)$ such that
\begin{equation*}
\begin{split}
&\exp[Y^{\Lambda,C}_t(\delta X)]=\exp(\delta X_0)\exp(\bar M_t -{1\over 2}\langle \bar M^c\rangle_t-(e^{\bar U}-\bar U-1).\nu_t+\bar A_t)
\\&\exp[Y^{\Lambda,C}_t(-\delta X)]=\exp(-\delta X_0)\exp(\underline{M}_t -{1\over 2}\langle \underline{M}^c\rangle_t-(e^{\underline{U}}-\underline{U}-1).\nu_t+\underline{A}_t)
\end{split}
\end{equation*}
then we find 
$\delta X_t+\Lambda_t+|X_t|*C_t=\delta X_0+\bar M_t -{1\over 2}\langle \bar M^c\rangle_t-(e^{\bar U}-\bar U-1).\nu_t+\bar A_t$ and 
$-\delta X_t+\Lambda_t+|X_t|*C_t=-\delta X_0+\underline{M}_t -{1\over 2}\langle \underline{M}^c\rangle_t-(e^{\underline{U}}-\underline{U}-1).\nu_t+\underline{A}_t$. 
Therefore $\bar M=-\underline{M}$ from uniqueness of the decomposition, moreover $\bar A_t+\underline{A}_t=\langle M^c\rangle+(e^{\bar U}-\bar U-1)\nu_t+(e^{-\bar U}+\bar U-1)\nu_t+2\Lambda_t +2|X_t|*C_t$. We deduce the both processes $\bar A$ and $\underline A$ are continuous and from Radon Nikodym Theorem, there exists a predictable process $0\le \alpha\le 2$ such that $dA_t={\alpha_t\over 2}\left[\langle M^c\rangle+(e^{\bar U}-\bar U-1)\nu_t+(e^{-\bar U}+\bar U-1)\nu_t+2\Lambda_t +2|X_t|*C_t\right]$ then we find the decomposition of $X=X_0-V+\widetilde M$ where:
$$d\widetilde V_t={\delta\over 2}(1-\alpha_t)d\langle \widetilde M^c\rangle_t+{(2-\alpha_t)\over 2}{1\over \delta}dj_t(\delta \Delta \widetilde M_t)+{(2-\alpha_t)\over 2}{1\over \delta}d\Lambda_t+{(2-\alpha_t)\over 2}|X_t|dC_t-{\alpha_t\over 2}dj_t(-\delta\Delta \widetilde M_t)$$ 
\noindent Since the predictable process $0\le \alpha\le 2$, we find:
\begin{equation*}
 \left\lbrace
\begin{split}
& -{\delta\over 2}d\langle \widetilde M^c\rangle_t-{1\over \delta}d\Lambda_t-|X_t|.dC_t-{1\over \delta}dj_t[-\delta \Delta \widetilde M^d_t)]<\!\!< dV_t,\\
& dV_t<\!\!< {\delta\over 2}d\langle \widetilde M^c\rangle_t+{1\over \delta} d\Lambda_t+|X_t|.dC_t+{1\over \delta}dj_t[\delta \Delta \widetilde M^d_t)].
\end{split}
\right.
\end{equation*}
\noindent then $X$ is a ${\cal Q}(\Lambda,C,\delta)$- local semimartingale. equivalently, we can use the same arguments for the positive local submartingale $U^{\Lambda,C}(e^{\delta X})$ and $U^{\Lambda,C}(e^{\delta X})$ to find that the process $X$ is a ${\cal Q}(\Lambda,C,\delta)$-local semimartingale. Moreover since the finite variation part of  $\widetilde V$ is continuous, jumps come from the local martingale part. Hence, the process $X$ is a c\`adl\`ag local semimartingale.
 \end{proof}
 In all the rest of the paper, since from a multiplicative transformation (see Proposition \ref{transform}), we can transform the general class 
${\cal Q}(\Lambda,C,\delta)$ to the class ${\cal Q}(\Lambda,C,1)$. We can give all results in the class ${\cal Q}(\Lambda,C):={\cal Q}(\Lambda,C,1)$ without losing any generality. 
\subsubsection{Uniform Integrable ${\cal Q}(\Lambda,C)$- semimartingales} 
  We use the entropic submartingales to characterize the integrability condition for ${\cal Q}(0,0,1)$-class.  Given an fixed horizon time, we find in this part sufficient condition on the terminal condition to have uniform integrability of general local quadratic exponential semimartingales. First, let give some generalization of entropic submartingales for general ${\cal Q}(\Lambda,C)$-semimartingales.
\begin{Theorem}\label{generalui} let $X$ be a c\`adl\`ag process and $T$ a fixed horizon time. 
\begin{enumerate}
\item Assuming,  $\exp(|X_T|)\in L^1$, the process $X$ is a ${\cal Q}(\Lambda,C)$-semimartingale which belongs to ${\cal D}_{exp}$ if and only if for any stopping times $\sigma\le \tau\le T$:
\begin{equation}\label{entropiesinequalities1}
X_{\sigma}\le \rho_\sigma(X_\tau+\Lambda_{\sigma,\tau}+|X|*C_{\sigma,\tau}) \hbox{ and } -X_{\sigma}\le \rho_\sigma(-X_\tau+\Lambda_{\sigma,\tau}+|X|*C_{\sigma,\tau}).
\end{equation}
\item Assuming $U^{\Lambda,C}_T(e^{|X|})\in L^1$, the process $X$ is a ${\cal Q}(\Lambda,C)$-semimartingale which belongs to ${\cal D}_{exp}$ if and only if for any stopping times $\sigma\le \tau\le T$:
$$X_\sigma\le \rho_\sigma({X_\tau}+\Lambda_{\sigma,\tau}+|X|*C_{\sigma,\tau}) \hbox{ and } -X_\sigma\le \rho_\sigma({-X_\tau}+\Lambda_{\sigma,\tau}+|X|*C_{\sigma,\tau}).$$
\end{enumerate}
\end{Theorem} 
 \begin{proof}
\begin{enumerate} 
\item Let $X$ a ${\cal Q}(\Lambda,C)$-semimartingales which belongs to the class ${\cal D}_{exp}$. From Theeorem \ref{generalsubmartingale}, $\exp(Y^{\Lambda,C}(X))$ and  $\exp(Y^{\Lambda,C}(-X))$ are submartingales which belong to the class ${\cal D}$. Therefore for any stopping times $\sigma\le \tau\le T$:
 $$\exp(Y^{\Lambda,C}_\sigma(X)\le \mathbb{E}\left[\exp(Y^{\Lambda,C}_\tau(X)\vert {\cal F}_\sigma\right] \hbox{ and } \exp(Y^{\Lambda,C}_\sigma(-X)\le \mathbb{E}\left[\exp(Y^{\Lambda,C}_\tau(-X)\vert {\cal F}_\sigma\right]$$
\noindent then the ${\cal Q}(\Lambda,C)$ semimartingale $X$ satisfies the entropy inequalities \eqref{entropiesinequalities1}. Let assume the inequalities \eqref{entropiesinequalities1} are satified then we conclude $\exp(Y^{\Lambda,C}(X))$ and  $\exp(Y^{\Lambda,C}(-X))$ are submartingales which belong to the class ${\cal D}$ then from Theorem \ref{generalsubmartingale}, $X$ is a ${\cal Q}(\Lambda,C)$ semimartingales which belong to the class ${\cal D}_{exp}$.
\item We use the same arguments  with the positive submartingales $U^{\Lambda,C}(e^X)$ and $U^{\Lambda,C}(e^{-X})$.  
\end{enumerate}
\end{proof} 
The Theorem \ref{generalui} gives sufficient integrable condition for ${\cal Q}(\Lambda,C)$-semimartingale $X$ such that it belongs to the class ${\cal D}_{exp}$. We can find another condition using the transformation $\bar Y^{\Lambda,C}(|X|)$ since it is a ${\cal Q}$-submartingale. Therefore, using the same arguments as assertions in Theorem \ref{generalui}, we find   $\bar Y^{\Lambda,C}(|X_t|) \le \bar\rho_t[\exp(\bar Y^{\Lambda,C}(|X_T|)]$ which is equivalent to the condition given by \cite{BElK11} in the continuous case (see Hypotehsis 2.8 \cite{BElK11}):
\begin{equation}\label{condition}
|X_t|\le \rho_t\left[e^{C_{t,T}}|Y_T|+\int_t^T e^{C_{t,s}}d\Lambda_s\right],\quad t\le T.
\end{equation}
This assumption is a necessary and sufficient condition for the process $\bar Y^{\Lambda,C}(|X|)$ to be in class ${\cal D}_{exp}$ (the proof is given in Lemma 2.9 of \cite{BElK11}). In the same way, assertions in Proposition 2.10 of \cite{BElK11} still hold since the authors give the result in the general case (without using the continuity of processes). Moreover using the same $L Log L$ Doob-inequality, we can find the same sufficient condition on the terminal value $\bar Y^{\Lambda,C}(|X|)$ such that $|X| \in {\cal D}_{exp}$.
\begin{Proposition}\label{LLogL} Let consider an fixed horizon time $T>0$ and let $L$ be a positive submartingale such that $\max L_T:=\max_{t\in[0,T]} L_t\in(1,+\infty)$. For any $m>0$, let $u_m$ the convex function defined on $\R^+$ defined by $u_m(x)=x-m-m\ln(x)$ and $u(x):=u_1(x)$, the following assertions are satisfied:
\begin{enumerate}
\item Using the Dol\'eans Dade representation of positive martingale $L$, $L_t={\cal E}(M^c_t+(e^U-1).\widetilde \mu)$, $t\le T$, we find:
$$H^{ent}:=\mathbb{E}[L_T\ln(L_T)]=\mathbb{E}\left[L_T\left({1\over 2}\langle M^c\rangle_T+(Ue^U-e^U+1).\nu_T\right)\right].$$
\item The following sharp inequality holds true:
$$u(\mathbb{E}(\max L_T))\le \mathbb{E}\left(L_T\ln(L_T)\right).$$
Moreover, if $L$ is a positive ${\cal D}$-submartingale, the previous inequality becomes:
$$u_m(\mathbb{E}(\max L_T))-u_m(L_0)\le \mathbb{E}[L_T\ln(L_T)]-\mathbb{E}(L_T)\ln\left[\mathbb{E}(L_T)\right].$$
\noindent where $m=\mathbb{E}(L_T)$.
\end{enumerate}   
\end{Proposition}

\begin{proof}:
\begin{enumerate}
\item To prove the assertion, let us first prove that the equality $$\mathbb{E}(\max L_T)-1=\mathbb{E}\left[L_T\ln(\max L_T)\right]$$ holds  true  in  our case. From Dellacherie \cite{Del79} p.375, $\max L_t(\omega)=L_t(\omega)$ for every jump time $t$ or every increasing of right of $s\longrightarrow \max L_s(\omega)$. Therefore $L=\max L$ on the right support of $d\max L$. Therefore we find $\max L_t=1+\int_0^t d\max L_s=\int_0^t {L_s\over \max L_s}d\max L_s$ then
$\mathbb{E}(\max L_T)-1=\mathbb{E}\left[L_T\ln(\max L_T)\right].$ holds true. From this equality, it is sufficient that $\max L_T \in L^1$ to find  $L_T\ln(L_T)\in L^1$. Let assume, $\max L_T\in L^1$ and let define the stopping times $T_K$ such that the positive local martingale $L_t={\cal E}(M_t+(e^U-1)\widetilde \mu_t)\le K$. The stopping times $T_K$ is increasing and goes to infinity with $K$. Let define the process $N^{\mathbb{Q}}=M^c-\langle M^c\rangle+U.(\widetilde \mu-(e^U-1).\nu).$ is a martingale with respect to $\mathbb{Q}=L_T\mathbb{P}$ and we get:
\begin{equation*}
\begin{split}
&\mathbb{E}\left[L_T\left({1\over 2}\langle M^c\rangle_T+(e^U-U-1).\nu_T\right)\right]=\lim_K\mathbb{E}\left[L_T\left({1\over 2}\langle M^c\rangle_{T\wedge T_K}+(e^U-U-1).\nu_{T\wedge T_K}\right)\right]\\&=\lim_K\mathbb{E}\left[L_{T\wedge T_K}\left({1\over 2}\langle M^c\rangle_{T\wedge T_K}+(e^U-U-1).\nu_{T\wedge T_K}\right)\right]
\end{split}
\end{equation*}
Since $ \mathbb{E}(L_{T\wedge T_K}N^{\mathbb{Q}}_{T\wedge T_K})=0$, we find:
$$\mathbb{E}\left[L_{T\wedge T_K}\ln(L_{T\wedge T_K})\right]=\mathbb{E}\left[L_{T\wedge T_K}\left({1\over 2}\langle M^c\rangle_{T\wedge T_K}+U(e^U-1).\nu_{T\wedge T_K}+(e^U-U-1).\nu_{T\wedge T_K}\right)\right]$$
\noindent We have $\mathbb{E}\left[L_{T\wedge T_K}\ln(L_{T\wedge T_K})\right]\le \mathbb{E}\left[L_{T\wedge T_K}\ln(\max L_{T\wedge T_K})\right]\le \mathbb{E}\left[\max L_T\right]-1\le +\infty$, then we get the result by taking the limit when $K$ goes to infinity.
\item The proof is done in \cite{BElK11}, since authors used the first assertion to prove the result. 
\end{enumerate}
\end{proof}
Let $X$ be a ${\cal Q}(\Lambda,C)$-semimartingale, applying the result of Proposition \ref{LLogL} to the positive submartingale $\exp(\bar Y^{\Lambda,C}(|X|))$, we conclude if $\mathbb{E}\left(\bar Y^{\Lambda,C}_T(|X|)\exp[\bar Y^{\Lambda,C}_T(|X|)]\right)\in L^1$ then we have $\max\mathbb{E}\left(\bar Y^{\Lambda,C}_T(|X|)\exp[\bar Y^{\Lambda,C}_T(|X|)]\right)\in L^1$ and the inequality \eqref{condition} is satisfied, therefore $\bar Y^{\Lambda,C}(|X|)$ belongs to  class ${\cal D}_{exp}$. To conclude this part, let recall the definition of the class of ${\cal Q}(\Lambda,C)$-semimartingales which belong to ${\cal D}_{exp}$ given by \cite{BElK11}.

\begin{Definition} Let $\eta_T$ be a ${\cal F}_T$-random variable such that $$\exp[\gamma \bar Y^{\Lambda,C}_T(|\eta_T|)]=\exp[\gamma(e^{C_T}|\eta_T|+\int_0^T e^{C_s}d\Lambda_s)]$$ belongs to $L^1$, for all $\gamma>0$.  We define a class of ${\cal S}_{Q}(|\eta_T|,\Lambda,C)$  of ${\cal Q}(\Lambda,C)$-semimartingales  $X$ such that $$|X_t|\le \bar\rho_t\left[e^{C_{t,T}}|\eta_T|+\int_t^T e^{C_{t,s}}d\Lambda_s\right], \quad a.s.$$
\end{Definition}
\section{Quadratic-exponential variation and stability result}
\subsection{A priori estimates}
We now focus  on the estimate  of the martingale part of a semimartingale $X \in  {\cal S}_Q(|\eta_T|,\Lambda,C)$. The estimates of the discontinuous martingales part allow us to conclude the predictable projection $j(\gamma \Delta M^d)$, $\gamma\in \{-1,1\}$ is well defined when the semimartingale $X$ lives in a suitable space. 

\begin{Proposition}\label{aprioriestimates}
Let consider a semimartingale $X\in {\cal S}_Q(|\eta_T|,\Lambda,C)$ which follows the decomposition $X=X_0-V+M^c+M^d$, where there exists a process $U\in {\cal G}_{loc}(\mu)$ such that $M^d=U.\widetilde \mu$ then  the matingales $\bar M=M^c+(e^U-1).\widetilde \mu$ and $\underline{M}=-M^c+(e^{-U}-1).\widetilde \mu$ belong to ${\cal U}_{exp}$ and ${\cal M}^p_0$ for any $p\ge 1$. 
\\\\\noindent 
Moreover if for any stopping times $\sigma\le T$ there exists a constant $c>0$ such that  $$\mathbb{E}\left[\exp(e^{C_{T}}|\eta_T|+\int_0^T e^{C_{s}} d\Lambda_s)\vert {\cal F}_\sigma\right]\le c,$$ then the processes $\bar M$ and $\underline{M}$ are $\rm{BMO}$ martingales.

\end{Proposition}

\begin{proof}
\begin{enumerate}
\item Let $X\in {\cal S}_Q(|\eta_T|,\Lambda,C)$, from Proposition \ref{generalsubmartingale}, $Y^{\Lambda,C}(X)=X+\Lambda+|X|*C$ and $Y^{\Lambda,C}(-X)$ are ${\cal Q}$-local submartingale. Moreover let recall the process $\bar Y^{\Lambda,C}(X)=e^C.|X|+e^C*\Lambda$ satisfies $Y^{\Lambda,C}(X)\le  \bar Y^{\Lambda,C}(X)$ and $Y^{\Lambda,C}(-X)\le  \bar Y^{\Lambda,C}(X)$, therefore since $X\in {\cal S}_Q(|\eta_T|,\Lambda,C)$, for any $p\ge 1$, we find:
\begin{equation}\label{bornes}
\exp(p|Y^{\Lambda,C}_t(|X|)|)\le\exp[p\bar Y^{\Lambda,C}_t(X)]\le \mathbb{E}\left[\exp[p(e^{C_T}|\eta_T|+\int_0^T e^{C_s}d\Lambda_s)]\vert {\cal F}_t\right]
\end{equation}
We conclude:
\begin{equation}\label{bornes1}
\mathbb{E}\left[\sup_{t\le T}\exp(p|Y^{\Lambda,C}_t(|X|)|)\right]<+\infty.
\end{equation}
From the submartingale property of $\exp(Y^{\Lambda,C}(X))$ and $\exp(Y^{\Lambda,C}(-X))$, from Yoeurp-Meyer decomposition, there exist increasing processes $\bar A$ and $\underline{A}$ such that:
\begin{equation*}
\begin{split}
&\bar K_t:=\exp(Y^{\Lambda,C}_t(X))=\exp(X_0){\cal E}(\bar M_t)\exp(\bar A_t)\\
&\underline{K}_t:=\exp(Y^{\Lambda,C}_t(-X))=\exp(-X_0){\cal E}(\underline{M}_t)\exp(\underline{A}_t)
\end{split}
\end{equation*}
Since $\bar A$ and $\underline{A}$ are increasing, from \eqref{bornes1} we conclude $\bar Z:={\cal E}(\bar M)$ and $\underline{Z}:={\cal E}(\underline{M})$ are uniformly integrable then $\bar M$ and $\underline{M}$  $\in {\mathcal U}_{\exp}$. Moreover $\bar Z$ and $\underline{Z}$ belong to ${\cal M}^p$, for any $p\ge 1$. Using intergration by part formula we find $d\bar K_t=\bar K_{t^-}\left[d\bar A_t+d\bar M_t\right]$ and $d\underline{K}_t=\underline{K}_{t^-}\left[d\underline{A}_t+d\underline{M}_t\right]$, that leads to $d[\bar K]_t=\bar K^2_{t^-}d[\bar M]_t \hbox{ and  } d[\underline{K}]_t=\underline{K}^2_{t^-}d[\underline{M}]_t.$ Therefore we find for any stopping times $\sigma\le T$, $  [\bar{M}]_{\sigma,T}=\int_\sigma^T { d[\bar{K}]_t\over \bar{K}^2_{t^-}}       \hbox{  and    }   [\underline{M}]_{\sigma,T}=\int_\sigma^T { d[\underline{K}]_t\over \underline{K}^2_{t^-}}$ then we find:

\begin{equation}\label{eqK}
[\bar M]_{\sigma,T}\le \sup_{\sigma\le t\le T}\left({1\over \bar K^2_t} \right)\times [\bar K]_{\sigma,T} \quad \hbox{    and     } \quad[\underline M]_{\sigma,T}\le \sup_{\sigma\le t\le T}\left({1\over \underline{K}^2_t}\right)\times[\underline K]_{\sigma,T}
\end{equation}
However we find a priori estilmates of $[\bar K]_T$ and $[\underline K]_T$ using Ito's decomposition of the submartingales $\bar K^2$ and $\underline{K}^2$:
\begin{equation*}
\begin{split}
&d\bar K^2_t=2 \bar K_{t^-}d\bar K_t+d[\bar K]_t=2\bar K^2_{t^-}[d\bar M_t+d\bar A_t]+d[\bar K]_t\\
&d\underline{K}^2_t=2 \underline{K}_{t^-}d\underline{K}_t+d[\underline{K}]_t=2\underline{K}^2_{t^-}[d\underline{M}_t+d\underline{A}_t]+d[\underline{K}]_t
\end{split}
\end{equation*}
Therefore for any stopping times $\sigma\le T$, we find:
\begin{equation}\label{Kbornes}
\begin{split}
\mathbb{E}\left[[\bar K]_{\sigma,T}\vert {\cal F}_\sigma\right]\le \mathbb{E}\left[\bar K^2_T\vert {\cal F}_\sigma\right]\quad \hbox{ and }\quad 
\mathbb{E}\left[[\underline K]_{\sigma,T}\vert {\cal F}_\sigma\right]\le \mathbb{E}\left[\underline K^2_T\vert {\cal F}_\sigma\right]
\end{split}
\end{equation}
Since $\sup_{0\le t\le T}\bar K_t$ and $\sup_{0\le t\le T}\underline K_t$ belong to $\mathrm{L}^p$, for any $p\ge 1$, from Garsia and Neveu Lemma (see \cite{BElK11} Lemma 3.3) we get:
\begin{equation}\label{KLp}
\begin{split}
\mathbb{E}\left[[\bar K]_{T}]^{p}\right]<+\infty\quad \hbox{ and } \quad  \mathbb{E}\left[[\underline K]_{T}]^{p}\right]<+\infty,\quad \forall p\ge 1
\end{split}
\end{equation}
Since $\sup_{0\le t\le T}{1\over \bar K_t}$ and $\sup_{0\le t\le T}{1\over \underline K_t}$ belong to $\mathrm{L}^p$ for any $p\ge 1$ and using \ref{KLp}, from \ref{eqK} we conclude using Cauchy Schwartz inequalities that for any $p\ge 1$:
\begin{equation*}
\mathbb{E}\left[[\bar M]^p_T\right]\le +\infty \quad \hbox{ and } \mathbb{E}\left[[\underline M]^p_T\right]\le +\infty
\end{equation*}
then using BDG inequalities, we conclude $\bar M$ and $\underline{M}$ belong to ${\cal M}^p_0$. Moreover if there exists a non negative constant $c$ such that
 $$\mathbb{E}\left[\exp(e^{C_{T}}|\eta_T|+\int_0^T e^{C_{s}} d\Lambda_s)\vert {\cal F}_\sigma\right]\le c,$$
then from  \ref{bornes} the processes $\bar K$ and $\underline K$ are bounded, using  \ref{eqK} and \ref{Kbornes} we conclude the martingales $\bar M$ and $\underline M$ are $\mathrm{BMO}$-martingales. 

\end{enumerate}

\end{proof}
\subsection{Stability results of quadratic exponential semimartingale}
Here, we  { present stability results for quadratic exponential
semimartingales which we shall use  for the construction of the
maximal solution of a class of quadratic BSDE's with jumps. We   first recall
a general stability theorem of Barlow and Protter \cite{BP90} for
a sequence of c\`adl\`ag special semimartingales} converging uniformly in
$\mathrm{L}^1$. We denote by $ X^* := \sup_{0\leq t \leq T}
|X_t|$.
\begin{Theorem}\label{theorem:BP90} Let $X^n$ be a
sequence of special semimartingales which {belongs } to $\mathcal{H}^1$
with canonical decomposition $ X^n = X_0^n + M^n - V^n$, and
{satisfies}:
\begin{equation}
\label{uniform:variation}
\mathbb{E} \Big[\int_0^T |dV_s^n| \, \Big] \leq C, \quad \mbox{and} \quad \mathbb{E} \, \big[ \big(M^n \big)^* \big] \leq C
\end{equation}
for some positive constant $C$. Assume that:
$$ \mathbb{E} \, \big[ \big(X^n - X\big)^* \big] \longrightarrow 0, \quad  \mbox{as} \quad  n \to \infty,$$
{where  $ X$ is an adapted process}, then $ X$ is a semimartingale
in $\mathcal{H}^1$  {with  canonical decomposition} $ X = X_0 + M
-V$ satisfying:
\begin{equation}
\label{estimate:variation}
\mathbb{E} \Big[\int_0^T |dV_s| \, \Big] \leq C, \quad \mbox{and} \quad \mathbb{E} \, \big[ \big(M \big)^* \big] \leq C
\end{equation}
and we have
\begin{equation}
\label{limite:variation}
\lim_{ n \to \infty} \mathbb{E} \, \big[ \big(V^n - V\big)^* \big] = 0 \quad \mbox{and} \quad  \lim_{ n \to \infty} \|M^n - M\|_{\mathcal{H}^1} = 0.
\end{equation}
\end{Theorem}
\begin{Lemma}\label{convergence} Let $X^n$ a sequence of ${\cal S}_Q(|\eta_T|,\Lambda,C)$ semimartingales which canonical decomposition $X^n=X^n_0-V^n+M^n$ which converge  in ${\cal H}^1$ to some process $X$. Therefore the process $X$ which canonical decomposition $X=X_0-V+M$ is an adapted c\`adl\`ag process which belongs to ${\cal S}_Q(|\eta_T|,\Lambda,C)$
such that:
\begin{equation*}
\label{limite:variation}
\lim_{ n \to \infty} \mathbb{E} \, \big[ \big(V^n - V\big)^* \big] = 0 \quad \mbox{and} \quad  \lim_{ n \to \infty} \|M^n - M\|_{\mathcal{H}^1} = 0.
\end{equation*}

\end{Lemma}

\begin{proof} Let consider $X^n=X^n_0-V^n+{(M^c)}^n+U^n.\widetilde \mu$ a sequence of ${\cal S}_Q(|\eta_T|,\Lambda,C)$ semimartingales. Firstly let prove that if the sequence $X^n$ converge to the process $X$ then this limit belongs to the space ${\cal S}_Q(|\eta_T|,\Lambda,C)$. The sequence $X^n\in {\cal S}_Q(|\eta_T|,\Lambda,C)$, hence for each $n\in \N$ and for any stopping times $\sigma \le T$:
\begin{equation}\label{SeqX}
-\bar\rho_\sigma(-\eta_T+\Lambda_{\sigma,T}+|X^n|*C_{\sigma,T})\le X^n_\sigma\le \bar\rho_\sigma(\eta_T+\Lambda_{\sigma,T}+|X^n|*C_{\sigma,T}),\quad a.s
\end{equation}
\noindent and we get also:
\begin{equation}\label{BoundSeqX}
|X^n_\sigma|\le \bar\rho_\sigma\left[e^{C_{\sigma,T}}|\eta_T|+\int_\sigma^T e^{C_{\sigma,s}}d\Lambda_s\right],\quad a.s
\end{equation}
\noindent Since the sequence $X^n$ converges in ${\cal H}^1$, we can extract a subsequence which converges uniformly almost surely to the limit $X$, using the dominated convergence and taking the limit of this subsequence in \ref{SeqX} and \ref{BoundSeqX}, we conclude the limit $X$ belongs to the space ${\cal S}_Q(|\eta_T|,\Lambda,C)$. 
Since the limit $X$ is a ${\cal Q}(\Lambda,C)$-semimartingale, from Theorem \ref{generalsubmartingale} the process $X$ is c\`adl\`ag. Let now prove the convergence of the martingale part and finite variation part,  the priori estimates of their limits.  
From the first assertion of Proposition \ref{aprioriestimates}, there exists a constant $C_1>0$ such that:
$$\mathbb{E}\left[\int_0^T|dV^n_s|\right]\le \mathbb{E}\left[\int_0^T {1\over 2}d\langle {(M^c)}^n\rangle_s+\int_E[g(U^n(s,x))+g(-U^n(s,x))]\nu(ds,dx)\right]\le C_1.$$
\noindent Moreover From the second assertion of Proposition \ref{aprioriestimates}, using BDG inequalities, there exists  constants $c_2$ and $C_2>0$ such that:
$$\mathbb{E} \big[ \big(M \big)^*_T \big]\le c_2\mathbb{E}\left[ \langle M^n\rangle^{{1\over 2}}_T\right] \le C_2.$$ 
\noindent Therefore assuming the sequence $X^n$ converges to a process $X$ which canonical decomposition $X=X_0-V+M$, taking $C=\max(C_1,C_2)$, from Barlow and Protter Theorem \ref{theorem:BP90}, we get the expected result.
 \end{proof}

\section{Application of quadratic exponential semimartingales: quadratic BSDEs with jumps}
In stochastic optimization problem with exponential utility (see \cite{Bech06}, \cite{Mor07}), or robust optimization with entropy penality see(\cite{BMS07}, \cite{JMN10}), using dynamic programing , the authors solved the problem using quadratic Backward Stochastic Differential Equations (BSDE) with bounded terminal condition in most of these papers. In the papers \cite{BElK11} and \cite{BriandHu07}, the authors dealt with the problem with unbounded terminal condition but they worked in continuos filtration. In this part, we use quadratic exponential semimartingales to find the solution of a quadratic BSDE with jumps in a general set up where the terminal condition is unbounded.

\subsection{Quadratic Exponential BSDE}
We consider the stochastic basis defined above $(\Omega,{\mathcal F},\mathbb{F},\mathbb{P})$ with finite time horizon $T<+\infty$. On this basis, let $W={(W_t)}_{t\ge 0}$ be a $d$-dimensional standard Brownian motion and let $\mu$ the random measure defined above such that $\nu$ is equivalent to a product measure $\lambda\otimes dt$ with density $\xi$ satisfying $\nu(\omega,dt,dx)=\xi(\omega,t,x)\lambda(dx)dt,$ where $\lambda$ is a $\sigma$-finite measure on $(E,{\cal E})$ satisfying $\int_E |x|^2\lambda(dx)<+\infty$ and where the density $\xi$ is a measure, bounded nonnegative function such that for some constant $C_\nu$:
$$0\le \xi(\omega,t,x)\le C_\nu<+\infty,\quad \P\otimes \lambda\otimes dt-\hbox{ a.e }.$$
\noindent That implies in particular $\nu\left([0,T]\times E\right)\le C_\nu T \lambda(E).$ We assume the following representation Theoem for any square integrable martingale $M$:
$$M=M_0+Z.W+U.\widetilde \mu,$$
\noindent where $Z$ and $U$ are predictable processes such that $\mathbb{E}\left[\int_0^T \big(|Z_t|^2+\int_E |U(t,x)|^2\xi(t,x)\lambda(dx)\big)dt\right]<+\infty.$ We note $\mathbb{H}^2$ resp($\mathbb{H}^2_\lambda$ ) the set of predictable process $Z$ resp( predictable process $U$) satisfying this square integrability condition. Let define the following norms $||.||_{\mathbb{H}^{2p}}$ and $||.||_{\mathbb{H}^{2p}_\lambda}$, for $p\ge 1$:
$$||Z||_{\mathbb{H}^{2p}}:={\left(\mathbb{E}\left[\int_0^T |Z_s|^2ds \right]^{p\over 2}\right)}^{1\over p}, \hbox{ for any predictable process } Z,$$
\noindent and 
$$||U||_{\mathbb{H}^{2p}_\lambda}:={\left(\mathbb{E}\left[\int_0^T |U|^2_{s,\lambda}ds \right]^{p\over 2}\right)}^{1\over p}, \hbox{ for any predictable process } U,$$
\noindent where \\\\\noindent 
$|U|^2_{s,\lambda}:=\int_E |U(s,x)|^2 \xi(s,x)\lambda(dx)$. 

\begin{Definition}(Quadratic Exponential BSDE) We call quadratic exponential BSDE associated to $(f,\eta_T)$ and  with parameters $(l,c,\delta)$ (in short terms we note $q_{exp}(l,c,\delta)$ BSDE), the following stochastic differential equation:
\begin{equation}\label{QBSDEjumps}
-dY_t=f(t,Y_t,Z_t,U_t)-Z_tdW_t-\int_E U(t,x)\widetilde \mu(dt,dx),\quad Y_T=\eta_T,
\end{equation}
\noindent where the coefficient satisfying the following conditions:
\begin{enumerate}
\item (\it{Continuity condition}): for all $t\in [0,T]$, $(y,z,u)\longrightarrow f(t,y,z,u)$ is continuous.
\vspace{2mm}
\item (\it{ Growth condition}): for all $(y,z,u)\in \R\times \R^d \times L^2$, $t\in [0,T]$: $\P$-a.s,
$$\underline{q}(t,y,z,u)={1\over \delta}j_t(-\delta u)-{\delta\over 2}|z|^2-l_t-c_t|y|\le f(t,y,z,u)\le {1\over \delta}j_t(\delta u)+{\delta\over 2}|z|^2+l_t+c_t|y|=\bar q(t,y,z,u).$$
\item \textbf{ $({\cal A}_{\gamma}$)} condition : there exists a process $\gamma$ such that for all $(y,z)\in \R\times \R^d$, $t\in [0,T]$: $\P$-a.s, 
\begin{equation}\label{Lipssaut}
f(t,y,z,u)-f(t,y,z,\bar u)\le \int_E \gamma_t[u(x)-\bar u(x)]\xi(t,x)\lambda(dx),
\end{equation}
\noindent where the process  $\gamma.\widetilde \mu$ belongs to the space ${\cal U}_{\exp}$ and $\gamma>-1$ .
\end{enumerate}
\end{Definition}

\noindent  A solution of a  Quadratic Exponential BSDE \eqref{QBSDEjumps}  is a   a triple $(Y,Z,U)$  of predictable  processes satisfying:

$$\int_0^T|Z^2_s|ds <+\infty,\quad  \int_0^T {(e^{U(t,x)}-1)}^2\xi(t,x)\lambda(dx)<+\infty \; \hbox{ and } \; \int_0^T {(e^{-U(t,x)}-1)}^2\xi(t,x)\lambda(dx) \hbox{  a.s  }$$
In particular case, where the coefficient satisfies a Lipschitz condition and the condition ${\cal A}_{\gamma}$ with $-1<\gamma<c$, for some positive constant $c$, there exists a unique solution of the BSDE see(\cite{Bech06}, \cite{Mor07}).
\begin{Theorem}\label{LipsTh} i) Let consider the BSDE \eqref{QBSDEjumps} with terminal value $Y_T=\eta_T\in L^2(\Omega,{\cal F}_T)$ where the coefficient satisfying the following conditions:
\begin{enumerate}
\item \it{Continuity and Integrability condition}: for all $t\in [0,T]$, $(y,z,u)\longrightarrow f(t,y,z,u)$ is continuous and satisfying the integrability condition:
\begin{equation}\label{Integ}
\mathbb{E}\left[\int_0^T |f(t,0,0,0)|^2dt\right]<+\infty.
\end{equation}
\vspace{2mm}
\item \it{ Lipschitz condition}: there exists a nonnegative constant $C$ such that for all $t\in [0,T]$:
\begin{equation}\label{Lipschitz condition}
 |f(t,y,z,u)-f(t,\bar y,\bar z,\bar u)|\le C\left[ |y-\bar y|+|z-\bar z|+|u-\bar u|_{t,\lambda}\right].
\end{equation}
\item \it{The bound $({\cal A}_{\gamma}$) condition}: the coefficient $f$ satisfies the (${\cal A}_{\gamma}$) condition and there exists a nonnegative constant $c$ such that $-1<\gamma\le c.$ 
\end{enumerate}
\noindent then there exists a unique triple $(Y,Z,U)\in{\cal S}^2\times \mathbb{H}^2\times \mathbb{H}^2_\lambda$ solution of the BSDE \eqref{QBSDEjumps}.\\\\\noindent 
ii) Moreover  for any BSDE with terminal value $\bar\eta_T\le \eta_T$ and coefficient $\bar f\le f$ satisfying the same last assumptons as $f$, the solution ($\bar Y,\bar Z,\bar U)$ associated to this BSDE satisfies $\bar Y\le Y$.
\end{Theorem}
\begin{Remark} (Comparison result)
\begin{enumerate}
\item The comparison result holds true also in the case of the both coefficients satisfies the Lipschitz and the bound $({\cal A}_\gamma)$ conditons.
\item The bound $({\cal A}_\gamma)$ condition can be substituted to the $({\cal A}_\gamma)$ condition since we need only $\gamma.\widetilde \mu$ should be in the space ${\cal U}_{\exp}$ to ensure the comparison result.
\end{enumerate}
\end{Remark}
\subsubsection{Existence of solution for quadratic BSDE with jumps}  
We prove the solution of the quadratic exponential BSDE using the properties of quadratic exponential semimartingales. we construct a sequence of quadratic exponential semimartingales which converges to some limit. Therefore using stability result, we find the convergence of the martingale part of the limit and conclude. 
\begin{Remark} Let consider the triple $(Y,Z,U)$ solution of the $q_{\exp}(l,c,\delta)$ BSDE, then $Y$ is a ${\cal Q}_{\exp}(\Lambda,C,\delta)$- semimartingale with $\Lambda_t=\int_0^t l_s ds$, $C_t=\int_0^t c_s ds$, $t\le T.$ In all the rest of paper without losing any generality, we assume $\delta=1$ and $c$ is bounded . 
\end{Remark}

\begin{Assumption}\label{Integcond} Let assume the integrability condition:
$$\forall\gamma>0,\quad \mathbb{E}\left[\exp(\gamma(e^{C_T}|\eta_T|+\int_0^T e^{C_s} d\Lambda_s))\right]<+\infty.$$
\end{Assumption}

\begin{Proposition}\label{Seq}(Construction of  sequences). Let assume \ref{Integcond} and consider the $q_{\exp}(l,c,1)$-BSDE associated to $(f,\eta_T)$.  We set $\bar f=f1_{f>0}$ and $\underline{f}=f1_{f\le 0}$ and we define for each $n,m\in \mathbb{N}$ the sequences of coefficients $\bar f^n=\bar f\vee\bar b^n$, $\underline{f}^m=\underline{f}\wedge\underline{b}^m$,  $\bar{q}^n=\bar q\vee\bar{b}^n$ and $\underline{q}^m=\underline{q}\wedge\underline{b}^m$   where the regularizing functions $\bar b^n$ and $\underline{b}^m$ are the convex functions with linear growth defined by $\bar{b}^n(w,r,v)=n|w|+n|r|+n|v|$ and $\underline{b}^m(w,r,v)=-m|w|-m|r|-m|v|$. The symbols $\vee$ and $\wedge$  stands for inf-convolution  and sup-convolution. 
\begin{enumerate}
\item The sequence $\bar f^n$ and $\bar q^n$  resp( the sequence $\underline{f}^m$ and $\underline{q}^m$) are increasing and converge to $\bar f$, $\bar q$ resp ( are  decreasing  and converge to $\underline{f}$, $\underline{q}$). Moreover the sequence $\underline{f}^m, \underline{q}^m$, $\bar f^n$ and $\bar q^n$ satisfy the Lipschitz condition \eqref{Lipschitz condition}. 
\item The sequences ${(\bar f^n)}_{n\in \N}$ resp( the sequence ${(\underline{f}^m)}_{m\in \N}$) satisfies $0\le \bar f^n\le \bar q^n\le \bar q$ (resp $\underline{q}\le \underline{q}^m\le \underline{f}^m\le 0)$, for each $n,m\in \N$. Moreover 
$$ \underline{q}\le \bar f^n+\underline{f}^m \le  \bar q.$$

 \end{enumerate}
\end{Proposition} 
\begin{proof}
 Let consider  the coefficient $f$ associated to the $q_{exp}(l,c,1)$ BSDE. Using the properties of infconvolution and supconvolution, we deduce all the sequences of coefficients defined above satisfy the Lipschitz condition, moreover we find the monotone property comparing using the definition of theses sequences. Let  prove the lower and upper bound of $\bar f^n$ and $\underline{f}^m$ for each $n,m\in \N$:
$$0\le \inf_{w,r,v}\{\bar{f}(t,w,r,v)+n|y-w|+n|z-r|+n|u-v|\}\le \inf_{w,r,v}\{\bar q(t,w,r,v)+n|y-w|+n|z-r|+n|u-v|\}.$$ 
\noindent then we find $0\le \bar{f}^n\le \bar{q}^n\le \bar q$. By similar arguments we find  for each $m\in \N$, $0\ge\underline{f}^m\ge \underline{q}^m\ge\underline{q}$, hence we conclude $0\le \bar {f}^n\le \bar q^n\le \bar q$ and $\underline{q}\le \underline{q}^m \le \underline{f}^m\le 0$. Moreover, using the last inequalities we deduce $\underline{q}\le f^n+f^m\le \bar q,$ for each $n,m\in \N.$
\end{proof}

\begin{Theorem}Let assume \ref{Integcond}, there exists a triple $(Y,Z,U)\in {\cal S}_Q(|\eta_T|,\Lambda,C)\times \mathbb{H}^{2 p} \times \mathbb{H}^{2 p}_\lambda $ solution of the $q_{\exp}(l,c,1)$ BSDE  associated to $(f,\eta_T)$.
\end{Theorem}

\begin{proof} We follow two steps to prove the existence. Firstly, we construct a sequence of quadratic exponential semimartingales which converges and secondly we find the convergence of the finite variation and martingale part using stability result. 
\\\\\noindent \textit{First step: (Construction of the sequence of ${\cal S}_Q(|\eta_T|,\Lambda,C)$ semimartingales)}.
Let consider a $q_{\exp}(l,c,1)$ BSDE associated to $(f,\eta_T)$ and consider the sequence of coefficients $f^{n,m}=\bar f^n+\underline{f}^m$ which converges to $f$ when $n,m$ go to infinity from Proposition \ref{Seq}. We consider the BSDE associated to $(f^{n,m},\eta_T)$:
$$-dY^{n,m}_t=f^{n;m}(t,Y^{n,m}_t,Z^{n,m}_t,U^{n,m}_t)dt-Z^{n,m}_tdW_t-\int_E U^{n,m}(t,x).\widetilde \mu(dt,dx),\quad Y^{n,m}_T=\eta_T.$$
\noindent Since for each $n,m \in \N$ the coefficient $f^{n,m}$ satisfies the continuity, the integrability and the Lipschitz conditions of Theorem \ref{LipsTh},  there exists a solution $(Y^{n,m},Z^{n,m},U^{n,m})$  of the BSDE associated to $(f^{n,m},\eta_T)$. From Proposition \ref{Seq}; $\underline{q}\le f^{n,m}\le \bar q$, hence ${(Y^{n,m})}_{n,m\in \N}$ is a sequence of ${\cal Q}(\Lambda,C)$ semimartingales. Moreover since $\underline{q}^m\le f^{n,m}\le \bar q^n$ and the triples $(\bar Y^n,\bar Z^n,\bar U^n)$ and $(\underline{Y}^m,\underline{Z}^m,\underline{U}^m)$ solutions of the $q_{exp}(l,c,1)$ BSDE associated to $(\bar q^n,|\eta_T|)$ and $(\underline{q}^m,-|\eta_T|)$ exist and satisfy for all stopping times $\sigma\le T$:
$$|\bar Y^n_\sigma|\vee |\underline{Y}^m_\sigma|\le \bar\rho_\sigma\left[e^{C_{\sigma,T}}|\eta_T|+\int_\sigma^T e^{C_{\sigma,s}}d\Lambda_s \right],\quad a.s$$
\noindent  The existence of the triples is given by existsence result of Theorem \ref{LipsTh}, since $\bar q^n$ and $\underline{q}^m$ satisfy the continuity, the integrability and the Lipschitz conditions  for each $n,m\in \N$. Moreover since they also satisfy  the $({\cal A}_\gamma)$-condition, we find by comparison result of Theorem \ref{LipsTh}  that  $\underline{Y}^m\le Y^{n,m}\le \bar Y^n$ and  we conclude  ${(Y^{n,m})}_{n,m\in \N}$ is a  sequence of ${\cal S}_{Q}(|\eta_T|,\Lambda,C)$ semimartingales since for each $n,m\in \N$, $Y^{n,m}$ is a ${\cal Q}(\Lambda,C)$ semimartingale satisfying: 
\begin{equation}\label{bound}
|Y^{n,m}_\sigma|\le \bar\rho_\sigma\left[e^{C_{\sigma,T}}|\eta_T|+\int_\sigma^T e^{C_{\sigma,s}}d\Lambda_s\right],\quad a.s
\end{equation}
\noindent for $n, m\ge c*=\sup_{t\le T} c_t$, see more details about the characteristics of the coefficient $\bar q^n$ and $\underline{q}^m$ in the Appendix Lemma 
\ref{ineqsaut}. Let now prove the coefficient satisfies the $({\cal A}_\gamma)$- condition. For $(y,z)\in \R\times \R^d$:
$$f^{n,m}(t,y,z,u)-f^{n,m}(t,y,z,\bar u)=[\bar f^n(t,y,z,u)-\bar f^n(t,y,z,\bar u)]+[\underline{f}^m(t,y,z,u)-\underline{f}^m(t,y,z,\bar u)].$$
\noindent Therefore since for any functions $\psi$ and $\bar \psi$:
\begin{equation}\label{InfSupinequality}
\begin{split}
&\inf_x \psi(x)-\inf_x \bar \psi(x)\le \sup_x \{\psi(x)-\bar \psi(x)\}, \quad 
\sup_x \psi(x)-\sup_x \bar \psi (x)\le \sup_x \{\psi(x)-\bar \psi(x)\} 
\end{split}
\end{equation}
\noindent Since the coefficient $f^{n,m}$ satisfies the $({\cal A}_\gamma)$ condition, we can apply comparison result see Theoem \ref{LipsTh}; we deduce for each $n,m\in \N$: $$Y^{n+1,m}\ge Y^{n,m}\ge Y^{n,m+1}.$$
\\\noindent\textit{ Second step: (Convergence of the semimartingale, the finite variation and the martingale part).}
For each $m\in \N$, ${(Y^{n,m})}_{n\ge 0}$ is an increasing sequence of bounded c\`adl\`ag ${\cal S}_Q(|\eta_T|,\Lambda,C)$ semimartingales, with canonical decomposition $Y^{n,m}=Y^{n,m}_0-V^{n,m}-M^{n,m}$. Hence, this sequence converges, let denote $Y^m$ its limit for each $m\in \N$. From stability result, Lemma \ref{convergence},
${(Y^m)}_{m}$ is a sequence of  c\`adl\`ag ${\cal S}_Q(|\eta_T|,\Lambda,C)$ semimartingales with canonical decomposition $Y^m=Y^m_0-V^m+M^m$ where 
 \begin{equation*}
\label{limite:variation}
\lim_{ n \to \infty} \mathbb{E} \, \big[ \big(V^{n,m} - V^m\big)^* \big] = 0 \quad \mbox{and} \quad  \lim_{ n \to \infty} \|M^{n,m} - M^m\|_{\mathcal{H}^1} = 0.
\end{equation*}
\noindent For each $n,m\ge c^*$, since $Y^{n,m}\ge Y^{n,m+1}$ then ${(Y^m)}_{m}$ is a decreasing sequence of bounded c\`adl\`ag ${\cal S}_Q(|\eta_T|,\Lambda,C)$ semimartingales. Let $Y$ its limit, from stability result Lemma \ref{convergence},
$Y$ is a  c\`adl\`ag ${\cal S}_Q(|\eta_T|,\Lambda,C)$ semimartingale with canonical decomposition $Y=Y_0-V+M$ where 
 \begin{equation*}
\label{limite:variation}
\lim_{ m \to \infty} \mathbb{E} \, \big[ \big(V^{m} - V\big)^* \big] = 0 \quad \mbox{and} \quad  \lim_{ m \to \infty} \|M^{m} - M\|_{\mathcal{H}^1} = 0.
\end{equation*}
 \noindent Let recall $dV^{n,m}_t=f^{n,m}(Y^{n,m}_t,Z^{n,m}_t,U^{n,m}_t)dt$ and consider the sequence of stopping times ${(T_K)}_{K\ge 0}$ defined by:
$$T_K=\inf \big\{ t \ge 0, \mathbb{E}\left[\exp( e^{C_{T}}|\eta_T|+\int_0^T e^{C_{s}}d\Lambda_s)\vert {\cal F}_t\right]>K\big\}$$
\noindent The sequence ${(T_K)}_{K\ge 0}$ converges to infinity when $K$ goes to infinity, moreover for $K>K_\epsilon$ large enough, $\mathbb{P}(T_K<T)\le{\epsilon\over K}$. From \ref{bound}, we find $Y^{n,m}_{.\wedge T_K}$ lives in a compact set and its convergence to the c\`adl\`ag process $Y$ is uniform. The same property holds for $M^{n,m}_{.\wedge T_K}$ and $V^{n,m}_{.\wedge T_K}$.  Let $Z^{n,m,K}_t=Z^{n,m}1_{t<T_K}$ and $U^{n,m,K}_t=U^{n,m}1_{t<T_K}$
in such that $(Z^{n,m}.W)_{.\wedge T_K}=Z^{n,m,K}.W$ and $(U^{n,m}.\widetilde\mu)_{.\wedge T_K}=U^{n,m,K}.\widetilde \mu$. Since the sequence $M^{n,m}_{.\wedge  T_K}=(Z^{n,m}.W)_{.\wedge T_K}+(U^{n,m}.\widetilde \mu)_{.\wedge T_K}$ strongly converges, the sequence of orthogonal martingales $(Z^{n,m,K}.W)$ and $(U^{n,m,K}.\widetilde \mu)$ also converge in their appropriate space. Therefore, we can extract a subsequence $Z^{n,m,K}$ and $U^{n,m,K}$ converging a.s to some  processes $Z$ and $U$. 
\par\medskip  For $t\le T_K$, the sequence $f^{n,m}(t,Y^{n,m}_t,Z^{n,m,K}_t,U^{n,m,K}_t)$ converges to $f(t,Y_t,Z_t,U_t) dt\otimes d\mathbb{P}$ a.s . it remains to prove that $\mathbb{E}\left[\int_0^{T_K} |f^{n,m}(t,Y^{n,m}_t,Z^{n,m}_t,U^{n,m}_t)-f(t,Y_t,Z_t,U_t)|dt \right]$ goes to zero when $n,m$ go to infinity. Firstly  we have  $$\mathbb{E}\left[\int_0^{T_K} |f^{n,m}(t,Y^{n,m}_t,Z^{n,m}_t,U^{n,m}_t)-f(t,Y_t,Z_t,U_t)|1_{\{[Z^{n,m}_t|+|U^{n,m}_t|\le C\}}dt \right]$$ goes to zero when $n,m$ go to infinity, by dominated convergence since $Y^{n,m}$ is bounded and $|f^{n,m}(t,Y^{n,m}_t,Z^{n,m}_t,U^{n,m}_t)-f(t,Y_t,Z_t,U_t)|$ is uniformly bounded in $L^1$ by Lemma \ref{convergence}. Moreover for $s\le T_K$,  $\mathbb{P}\left(|Z^{n,m}_s|+|U^{n,m}_s|>C\right)\le {2\over C^2}\mathbb{E}({|Z^{n,m}_s|}^2+{|U^{n,m}_s|}^2)$, from  Lemma \ref{convergence}, there exists a constants $C_2$ such that $\mathbb{E}(\langle M^{n,m}\rangle_s)\le C_2$, therefore $$\mathbb{E}\left[\int_0^{T_K} |f^{n,m}(t,Y^{n,m}_t,Z^{n,m}_t,U^{n,m}_t)-f(t,Y_t,Z_t,U_t)|1_{\{[Z^{n,m}_t|+|U^{n,m}_t|> C\}}dt \right]$$ goes to zero when $C$ goes to infinity, uniformly in $n,m$. As a consequence, the process $V$ in the decomposition of the quadratic exponential semimartingale $Y$ is given by $dV_t=f(t,Y_t,z_t,U_t)dt$ on $[0,T_K]$ for any $K$. We conclude the triple $(Y,Z,U)$ is a solution of the $q_{\exp}(l,c,1)$ BSDE associated to $(f,\eta_T)$. Moreover since $Y$ belongs to the space ${\cal S}_Q(\eta_T,\Lambda,C)$, then from Proposition \ref{aprioriestimates} the martingales $Z.W+(e^U-1).\widetilde \mu$ and $-Z.W+(e^{-U}-1).\widetilde \mu$ belongs to the space ${\cal M}^p_0$.  
\end{proof}

\newpage
\section{Appendix}
\begin{Lemma}\label{ineqsaut} For any $k\ge 1$ and any local martingale $M$:  $$j_t(k\Delta M_t)\ge k j_t(\Delta M_t),\quad 0\le t\le T$$
\end{Lemma}
\begin{proof} Let recall that for any local martingale $M=M^c+M^d$ from representation theorem, there exists $U\in{\cal G}_{loc}(\mu)$ such that $M^d=U.\widetilde \mu$ , then $j(\Delta M^d)=(e^{U}-U-1).\nu$. Therefore from representation theorem it is sufficient to prove the following function $f_k$: $x\longrightarrow (e^{kx}-kx-1)-k(e^x-x-1)$ is positive to find the result. For any $x\in \R$, since $f_k'(x)=ke^x(e^{(k-1)x}-1)$, then we conclude the function $f_k$ is increasing on $(0,+\infty)$ and decreasing on $(-\infty,0)$. Therefore, for any $x\in \R$, $f_k(x)\ge f_k(0)=0.$ 
\end{proof}



\begin{Lemma}\label{ineqsaut} Let us define the following quadractic exponential coefficients by $\bar q(y,z,u)=c|y|+|l|+{1\over 2}|z|^2+j(u)$ and $\underbar q(y,z,u)=-c|y|-|l|-{1\over 2}|z|^2-j(-u)$ and define the sequence $\bar q^n$ and $\underbar q^m$ by the inf-convolution and sup-convolution for $n,m\ge c^*=\sup_{t\in[0,T]} c_t$:

$$\bar q^n(y,z,u)=\inf_{r,w,v}\{\bar q(y,w,v)+n|y-r|+n|z-w|+n|u-v|\}$$
\noindent  
$$\underbar q^m(y,z,u)=\sup_{r,w,v}\{\underbar q(y,w,v)-m|y-r|-m|z-w|-m|u-v|\}$$
\noindent then:\\\\
\noindent i) The sequences $\bar q^n$ and $\underbar q^m$ satisfy the structure condition $\mathcal{Q}_{\exp}(\Lambda,C) $.\\\\ \noindent
\noindent ii) There exists a unique solution $(\bar Y^n, \bar Z^n,\bar U^n)$  (resp. ($(\underbar Y^m, \underbar Z^m,\underbar U^m)$) of the BSDE's associated to $(\bar q_n,|\xi_T|)$ (resp. to ($\underbar q^m,-|\xi_T|$).\\\\
\noindent iii) The processes $\bar Y^n$ and $\underbar Y^m$ are values processes of the following robust optimization problem,, for any $\sigma\le T$:

\begin{equation*}
\begin{split}
&\bar Y^n_\sigma=\sup_{\{\mathbb{Q}\ll\P, |\beta|\le n; -1\le \kappa\le n\}}\mathbb{E}^\mathbb{Q}_\sigma\left[S^c_{\sigma,T}|\xi_T|+\int_\sigma^T S^c_{\sigma,t} |l_t|dt+\int_\sigma^T c_t S^c_{\sigma,t}\ln\left({Z^\Q_t\over Z^\Q_\sigma}\right)dt+S^c_{\sigma,T}\ln\left({Z^\Q_T\over Z^\Q_\sigma}\right)\right]\\
&\underbar Y^m_\sigma=\inf_{\{\Q\ll\P, |\beta|\le m; -1\le \kappa\le m\}}\mathbb{E}^\Q_\sigma\left[-S^c_{\sigma,T}|\xi_T|-\int_\sigma^T S^c_{\sigma,t} |l_t|dt-\int_\sigma^T c_t S^c_{\sigma,t}\ln\left({Z^\Q_t\over Z^\Q_\sigma}\right)dt+S^c_{\sigma,T}\ln\left({Z^\Q_T\over Z^\Q_\sigma}\right)\right]
\end{split}
\end{equation*}
\noindent where  $S^c_{\sigma,t}=\exp(\int_\sigma^t c_s ds)$ and the Radon Nikodym density of $\Q$ with respect to $\P$ on ${\mathcal G}_T$ is $Z^{\Q}_T$, the process  $Z^\Q_t:=\mathbb{E}\left[Z^{\Q}_T\vert {\mathcal G}_t\right]={\cal E}(\beta.W+\kappa.\widetilde \mu)_t$.\\\\\noindent
Moreover, we have the following estimates:  $$ \underbar Y^m_\sigma\le Y^{n,m}_\sigma\le \bar Y^n_\sigma,\quad \sigma\le T.$$
\noindent and

$$|\underbar{Y}^m_\sigma|\vee |\bar Y^n_\sigma|\le \rho_\sigma\left[S^c_{\sigma,T}|\xi_T|+\int_\sigma^T S^c_{\sigma,s}|l_s|ds\right],\quad \sigma\le T.$$
\end{Lemma}

\noindent \textbf{Proof:} i) We compute explicitly the sequence of functions $\bar q^n$ and $\underbar q^m$ which satisfy  the  structure condition  $\mathcal{Q}_{\exp}(\Lambda,C) $ for every $n,m\ge c^*$. By  definition we have
\begin{equation*}
\begin{split}
\bar q^n_t(y,z,u)&=\inf_{r,w,v}\{\bar q_t(y,w,v)+n|y-r|+n|z-w|+n|u-v|\}\\&=c|y|+|l|+\inf_w \{{1\over 2}|w|^2+n|z-w|\}+\inf_v\{j_t(v)+n|u-v|\}
\end{split}
\end{equation*}
Obviously one can find the  explicit   form of $ \bar q^n$ which is given by

\begin{equation*}
\begin{split}
&\bar q^n(y,z,u)=c|y|+|l|+{1\over 2}|z|^2\textbf{1}_{\{|z|\le n\}}+n(|z|-{n\over 2})\textbf{1}_{\{|z|> n\}}\\&+\int_E\left[g(u(e))\textbf{1}_{\{e^{u(e)}-1\le n\}}+\big(-(n+1)\ln(n+1)+n(u(e)+1)\big)\textbf{1}_{\{e^{u(e)}-1>n\}}\right]\zeta_t(e)\rho(de)
\end{split}
\end{equation*}
where we recall $g(x)=e^x-x-1$,  using the similar arguments we find the explicitely form of $\underbar q_m$:
\begin{equation*}
\begin{split}
&\underbar q^m(y,z,u)=-c|y|-|l|-{1\over 2}|z|^2\textbf{1}_{\{|z|\le m\}}-m(|z|-{m\over 2})\textbf{1}_{\{|z|> m\}}\\&+\int_E\left[g(-u(e))\textbf{1}_{\{e^{-u(e)}-1\le m\}}+\big((m+1)\ln(m+1)+m(u(e)-1)\big)\textbf{1}_{\{e^{-u(e)}-1>m\}}\right]\zeta_t(e)\rho(de)\\
\end{split}
\end{equation*}
\noindent then we conclude for each $n,m\ge c^*$, $\bar q^n$ and $\underbar q^m$ satisfy the  structure condition $ \mathcal{Q}_{exp} (\Lambda,C)$.
\\\\\noindent ii) Since the coefficients $\bar q^n$ and $\underbar q^m$ are Lipschitz we deduce there exists a solution $(\bar Y^n,\bar Z^n,\bar U^n)$ resp $(\underbar Y^m,\underbar Z^m,\underbar{U}^m)$ associated to $(\bar q^n,|\xi_T|)$ resp $(\underbar q^m,-|\xi_T|)$. Let now prove these coefficients satisfy the (${\cal A}_{\gamma}$) condition. To prove this result let first remark that for all $x\in \R$:

$$-(n+1)\ln(n+1)+n(x+1)=g[\ln(n+1)]+n(x-\ln(n+1))$$
\noindent Let $u,\bar u$, we set $E=\bigcup_{i=1}^4 A_i$ where 
\begin{equation*}
\begin{split}
& A_1=\{ x \in \R, e^{u(x)}-1\le n, e^{\bar u(x)}-1\le n\},\quad  A_2=\{ x \in \R, e^{u(x)}-1\le n, e^{\bar u(x)}-1>n\}
\\
& A_3=\{ x \in \R, e^{u(x)}-1>n, e^{\bar u(x)}-1\le n\},\quad  A_4=\{ x \in \R, e^{u(x)}-1> n, e^{\bar u(x)}-1>n\}
\end{split}
\end{equation*}

\noindent Therefore we find for all $y,z$:
\begin{equation*}
\begin{split}
&\bar q^n(y,z,u)-\bar q^n(y,z,\bar u)\\&=\int_{A_1} [g(u(x))-g(\bar u(x))]\zeta_t(x)\rho(dx)+\int_{A_2} [g(u(x))-g(\ln(n+1))-n(\bar u(x)-\ln(n+1))]\zeta_t(x)\rho(dx)\\&+
\int_{A_3} [n(u(x)-\ln(n+1))+g(\ln(n+1))-g(\bar u(x))]\zeta_t(x)\rho(dx)+\int_{A_4} n(u(x)-\bar u(x))\zeta_t(x)\rho(dx)
\end{split}
\end{equation*}
We now find differents inequalities on every given subset of $E$:
\\\\\noindent \textit{On the set $A_1$}: we use the convex property of the coefffient $g$ and we deduce:
\begin{equation}\label{A1}
\int_{A_1} [g(u(x))-g(\bar u(x))]\zeta_t(x)\rho(dx)\le \int_{A_1} (e^{\bar u(x)}-1)(u(x)-\bar u(x))\zeta_t(x)\rho(dx)
\end{equation}
\textit{On the set $A_2$}:
Since on $A_2$, the function $g$ is increasing we find $\forall x\in A_2$, $g(u(x))\le g(\ln(n+1))$, moreover we have $\bar u(x)-\ln(n+1)\ge 0$ for $x\in A_2$ then we conclude:
\begin{equation}\label{A2}
\int_{A_2} [g(u(x))-g(\ln(n+1))-n(\bar u(x)-\ln(n+1))]\zeta_t(x)\rho(dx)\le 0.
\end{equation}
\textit{On the set $A_3$}: we use the convex property of the coefficient $g$ and we find $g(\ln(n+1))-g(\bar u(x))\le -n(\bar u(x)-\ln(n+1)),\quad \forall x \in A_3$.
Therefore, we find:
\begin{equation}\label{A3}
\int_{A_3} [n(u(x)-\ln(n+1)+g(\ln(n+1))-g(\bar u(x))]\zeta_t(x)\rho(dx)\le \int_{A_3} n(u(x)-\bar u(x))\zeta_t(x)\rho(dx)
\end{equation}
Hence using \eqref{A1}, \eqref{A2} and \eqref{A3}, we find:
\begin{equation*}
\begin{split}
&\bar q^n(y,z,u)-\bar q^n(y,z,\bar u)=\int_E \gamma^n(u(x),\bar u(x))(u(x)-\bar u(x))\zeta_t(x)\rho(dx)
\end{split}
\end{equation*}
where $\gamma^n(u(x),\bar u(x))=n \textbf{1}_\{e^{u(x)}-1>n\}+(e^{\bar u(x)}-1)1_{A_1}$, then $-1\le \gamma^n\le n$ hence $\gamma^n \in {\mathcal U}^{\exp}$. Therefore the sequence $\bar q^n$ satisfies the ${\mathcal A}_\gamma$ condition. We use similar arguments to prove the sequence $\underbar q^m$ satisfies also the  ${\mathcal A}_\gamma$ condition. We conclude from Comparison Theorem, the uniqueness of the triple $(\bar Y^n,\bar Z^n, \bar U^n)$ resp( $(\underbar{Y}^m,\underbar{Z}^m,\underbar{U}^m)$ solution of the BSDE associated to $(\bar q^n,|\xi_T|)$ resp($(\underbar{q}^m,-|\xi_T|)$).
\\\\\noindent iii) Let define the cost functional of the robust optimization problems defined in Lemma \ref{ineqsaut}-iii), $\bar J^{\Q,n}$ and $\underbar{J}^{\Q,m}$, and define the value processes $\bar V^n$, $\underbar V^m$. Assume $|\xi_T|\in \mathcal{L}^{\exp}$, $|l|\in {\cal D}^1_{\exp}$ and $c$ bounded then the value processes of the robust optimization exist see Bordigoni, Matoussi and Schweizer \cite{BMS07} for more details. Moreover we deduce for any $\Q\ll\P$:
\begin{equation}\label{relationJ}
\begin{split}
&\bar J^{\Q,n}_t=S^c_t \bar V^n_t+\int_0^t S^c_s|l_s|ds+\int_0^t c_s S^c_s\ln(Z^\Q_s)ds-S^c_t\ln(Z^\Q_t)\\&
\underbar{J}^{\Q,m}_t=S^c_t \underbar{V}^m_t-\int_0^t S^c_s|l_s|ds-\int_0^t c_s S^c_s\ln(Z^\Q_s)ds+S^c_t\ln(Z^\Q_t)
\end{split}
\end{equation}
\noindent moereover the value processes $\bar V^n$ and $\underbar V^m$ are special semimartingales the following the representation theorem there exist a predictable process $\bar Z^n$, $\bar U^n$ and an predicatble process $A^{\bar V^n}$ resp ( $\underbar Z^m$,$\underbar U^m$ and an predictable process $A^{\bar V^n}$) such that $d\bar V^n_t=dA^{\bar V^n}_t+\bar Z^n_t.dW_t+\int_E\bar U^n_t(e).\widetilde \mu(dt,de)$ resp( $d\underbar{V}^m_t=dA^{\underbar{V}^m}_t+\underbar {Z}^m_t.dW_t+\int_E\underbar{U}^m_t(e).\widetilde \mu(dt,de)$). we define the dynamics of $Z^\Q$, $$dZ^\Q_t=Z^\Q_{t^-}\left(\beta_t.dW_t+\int_E \kappa_t.\widetilde \mu(dx,dt)\right)$$
\noindent For every $n,m\in\N^*$, for every $\Q\ll\P$,  $\bar J^{\Q,n}$ resp( $\underbar{J}^{\Q,m}$ ) are submartingales and martingales for the optimal resp( surmartingales and martingale for the optimal) then we get:
\begin{equation*}
\begin{split}
A^{\bar V^n}_t&=-\max_{\{|\beta|\le n\}}\left\{\int_0^t(|l_s|+c_s\bar V^n_s)+\int_0^t \left(\langle\bar Z^n-\beta,\beta\rangle_s +{1\over 2}|\beta_s|^2\right)ds\right\}\\&
-\max_{\{ -1\le\kappa\le n\}}\left\{\int_0^t \int_{E}\left( \kappa_s(x)(v_s(x)+1)-(1+\kappa_s(x))\ln(1+\kappa_s(x))\right)\zeta_s(x)\rho(dx)ds\right\}
\end{split}
\end{equation*}
resp
\begin{equation*}
\begin{split}
A^{\underbar{V}^m}_t&=-\min_{\{|\beta|\le m\}}\left\{\int_0^t(-|l_s|+c_s\underbar V^m_s)+\int_0^t \left(\langle\underbar Z^m+\beta,\beta\rangle_s -{1\over 2}|\beta_s|^2\right)ds\right\}\\&
-\min_{\{ -1\le\kappa\le m\}}\left\{\int_0^t \int_{E}\left( \kappa_s(x)(v_s(x)-1)+(1+\kappa_s(x))\ln(1+\kappa_s(x))\right)\zeta_s(x)\rho(dx)ds\right\}
\end{split}
\end{equation*}
\noindent Using first order condition, we find for the first optimization problem $\kappa^*=\left(e^{\bar U^n}-1\right)\textbf{1}_{\{ e^{\bar U^n}-1\le n\}}+n\textbf{1}_{\{e^{\bar U^n}-1>n\}}$, then we deduce $(\bar V^n, \bar Z^n,\bar U^n)$ is the solution associated to the BSDE $(\bar h^n,|\xi_T|)$:
$$-d\bar V^n_t=\bar h_n(\bar V^n,\bar Z^n_t,\bar U^n_t)dt-\bar Z^n_t.dW_t-\bar U^n_t(x).\widetilde \mu(dt,dx),\quad \bar V^n_T=|\xi_T|.$$
\noindent where
\begin{equation*}
\begin{split}
&\bar h^n(y,z,u)=c y+|l|+{1\over 2}|z|^2\textbf{1}_{\{|z|\le n\}}+n(|z|-{n\over 2})\textbf{1}_{\{|z|> n\}}\\
&+\int_E \left[(e^{u(e)}-u(e)-1)\textbf{1}_{\{e^{u(e)}-1\le n\}}+\big(-(n+1)\ln(n+1)+n(u+1)\big)\textbf{1}_{\{e^{u(e}-1>n\}}\right]\zeta_t(e)\rho(de)
\end{split}
\end{equation*}
\noindent We use the same arguments of first order condition to deduce the solution of the second optimization problem; We get the triple $(\underbar V^m,\underbar Z^m,\underbar U^m)$ is the solution of the BSDE associated to $(\underbar h^m,-|\xi_T|)$:

$$-d\underbar V^m_t=\underbar h^m(\underbar{V}^m,\underbar{Z}^m_t,\underbar{U}^m_t)dt-\underbar{Z}^m_t.dW_t-\int_E\underbar{U}^m_t(x).\widetilde \mu(dx,dt),\quad \underbar {V}^m_T=-|\xi_T|.$$
\noindent where
\begin{equation*}
\begin{split}
&\underbar h^m(y,z,u)=cy-|l|-{1\over 2}|z|^2\textbf{1}_{\{|z|\le m\}}-m(|z|-{m\over 2})\textbf{1}_{\{|z|> m\}}\\
&+\int_E \left[g(-u(e))\textbf{1}_{\{e^{-u(e)}-1\le m\}}+\big((m+1)\ln(m+1)+m(u(e)-1)\big)\textbf{1}_{\{e^{-u(e)}-1>m\}}\right]\zeta_t(x)\rho(dx)\\
\end{split}
\end{equation*}
\noindent To finish the proof, we find $\bar Y^n\ge 0$ and  $\underbar Y^m\le 0$ since $|\xi_T|\ge 0$ and  $-|\xi_T|\le 0$ by comparaison theorem in Lipschitz case. Then we conclude $\bar h^n=\bar q^n$ and $\underbar h^m=\underbar q^m$.
for each $n,m\in \N$. By uniqueness of the solution of the BSDE associated to ($\bar q^n, |\xi_T|)$ and ($\underbar{q}^m, -|\xi_T|)$, we conclude $\bar V^n=\bar Y^n$ and $\underbar{V}^m=\underbar{Y}^m$ moreover since  $\underbar{q}^m \le f^{n,m}\le \bar q^n$ by Comparison Theorem, we find $\underbar{Y}^n\le Y^{n,m}\le \bar Y^n$.  However, from the dual representation of $\bar Y^n$ (resp $\underbar{Y}^m$) given by $\bar V^n$ resp( $\underbar{V}^m$), we conclude by Proposition 4.2 of \cite{BElK11} that for any $\sigma\le T$:
$$\bar Y^n_\sigma\le \rho_\sigma\left[S^c_{\sigma,T}|\xi_T|+\int_\sigma^T S^c_{\sigma,s}|l_s|ds\right] \hbox{ and } \quad \underbar {Y}^m_\sigma\le \rho_\sigma\left[S^c_{\sigma,T}|\xi_T|+\int_\sigma^T S^c_{\sigma,s}|l_s|ds\right]. $$
$\hfill\Box$\\[0.5cm]

\begin{Theorem}\label{Ito1} Let $X$ a semimartingale and $f\in {\cal C}^2$, then $f(X)$ is a semimartingale and we have:
\begin{equation*}
\begin{split}
f(X_t)=f(X_0)+\int_0^t {\partial f\over \partial x}(X_{s^-})dX_s+{1\over 2}\int_0^t {\partial^2 f \over \partial x^2}(X_{s^-})d\langle X^c\rangle_s
+\sum_{s\le t} \left[f(X_s)-f(X_{s^-})-{\partial f\over \partial x}(X_{s^-})\Delta X_s\right]
\end{split}
\end{equation*}
\noindent Moreover for any semimartingale  $X=X_0-V+M^c+M^d$, with continuous finite variation $V$, continuous martingale $M^c$ and jump martingale $M^d=U.(\mu-\nu)$ for some $U\in {\mathcal G}_{loc}$, we find:
\begin{equation*}
\begin{split}
f(X_t)&=f(X_0)+\int_0^t {\partial f\over \partial x}(X_{s^-})dX_s+{1\over 2}\int_0^t {\partial^2 f \over \partial x^2}(X_s^{-})d\langle M^c\rangle_s
\\&+\int_0^t\int_\R \left[f(X_{s^-}+U(s,x))-f(X_{s^-})-{\partial f\over \partial x}(X_{s^-})U(s,x)\right]\mu(ds,dx)
\end{split}
\end{equation*}

\end{Theorem}

\begin{Corollary}\label{Ito2} Let $f \in {\mathcal C}^2$ and consider the operators $D_1$ and $D_2$ defined by:
\begin{equation*}
D_1 f(x,y)= \left\lbrace
\begin{split}
& {f(x+y)-f(y)\over y},\quad \hbox{ if } y\not=0,\\
& f'(x),\quad \hspace{1.9cm}\hbox{ if } y=0 
\end{split}
\right.
\end{equation*}

and 

\begin{equation*}
D_2 f(x,y)= \left\lbrace
\begin{split}
& {2[f(x+y)-f(y)-yf'(y)]\over y^2},\quad \hbox{ if } y\not=0,\\
& f"(x),\quad \hspace{3.6cm}\hbox{ if } y=0 
\end{split}
\right.
\end{equation*}
For all $x,y \in \R$. then for any semimartingale $X$, the It\^o decomposition of $f(X)$ is given by:

$$f(X_t)=f(X_0)+\int_0^t D_1 f(X_s,0)dX_s+{1\over 2}\int_0^t D_2 f(X_s,\Delta X_s) d[X]_s.$$

\end{Corollary}

\end{document}